\documentclass[11pt,a4paper,reqno]{amsart}

\usepackage[T1]{fontenc}
\usepackage[utf8]{inputenc}
\usepackage[english]{babel}
\usepackage[margin=26mm]{geometry}
\usepackage{amsmath,amssymb,amsthm,mathtools,mathrsfs}
\usepackage{booktabs,array,tabularx,microtype}
\usepackage{hyphenat}
\usepackage[bookmarksopen=true,bookmarksnumbered=true]{hyperref}
\usepackage{fancyhdr}
\usepackage{tikz}

\allowdisplaybreaks
\newtheorem{theorem}{Theorem}[section]
\newtheorem{lemma}[theorem]{Lemma}
\newtheorem{proposition}[theorem]{Proposition}
\newtheorem{corollary}[theorem]{Corollary}
\theoremstyle{definition}\newtheorem{definition}[theorem]{Definition}
\newtheorem{openproblem}[theorem]{Open problem}
\theoremstyle{remark}

\newcommand{\R}{\mathbb R}
\newcommand{\BCC}{\mathrm{BCC}}
\newcommand{\FCC}{\mathrm{FCC}}
\newcommand{\KL}{D_{\mathrm{KL}}}
\newcommand{\trans}{\mathsf T}
\newcommand{\cP}{\mathcal P}
\newcommand{\cJ}{\mathcal J}
\newcommand{\Vor}{\operatorname{Vor}}
\newcommand{\Iso}{\operatorname{Iso}}
\newcommand{\adj}{\operatorname{adj}}
\newcommand{\tr}{\operatorname{tr}}
\newcommand{\covol}{\operatorname{covol}}

\newcommand{\doi}[1]{\href{https://doi.org/#1}{\nolinkurl{doi:#1}}}

\title{The truncated octahedron minimizes surface area \\
among parallelohedra of equal volume}
\author{Annalisa Cesaroni}
\address{Dipartimento di Matematica ``Tullio Levi-Civita'', Universit\`a di Padova, Via Trieste 63, 35131 Padova, Italy}
\email{annalisa.cesaroni@unipd.it}
\author{Matteo Novaga}
\address{Dipartimento di Matematica, Universit\`a di Pisa, Largo Bruno Pontecorvo 5, 56127 Pisa, Italy}
\email{matteo.novaga@unipi.it}

\hypersetup{
  pdftitle={The truncated octahedron minimizes surface area among three-dimensional parallelohedra},
  pdfauthor={Annalisa Cesaroni and Matteo Novaga},
  pdfsubject={Isoperimetric inequalities for three-dimensional lattice Voronoi cells and parallelohedra},
  pdfkeywords={parallelohedron, truncated octahedron, lattice Voronoi cell, zonotope, BCC lattice, Selling parameters, Ball--Barthe inequality, determinantal process}
}

\begin{document}
\begin{abstract}
We prove that the regular truncated octahedron uniquely minimizes surface area
among all parallelohedra of fixed volume.  Equivalently,
every three-dimensional parallelohedron $P$ satisfies
\[
 \frac{\mathcal H^2(\partial P)}{|P|^{2/3}}
 \ge \frac{3(1+2\sqrt3)}{4^{2/3}},
\]
with equality if and only if $P$ is similar to the regular truncated
octahedron.  Among the non-truncated Fedorov types we prove a 
stronger sharp bound, attained uniquely by the regular rhombic
dodecahedron.
\end{abstract}

\maketitle
\tableofcontents

\section{Introduction}

A three-dimensional parallelohedron is a convex polytope whose translates tile
$\mathbb R^3$ face-to-face.  In this paper we study the isoperimetric problem
within this class: given a parallelohedron $P$, we consider the scale-invariant
quantity
\begin{equation}\label{eq:isoperimetric-quotient}
 \Iso(P)=\frac{\mathcal H^2(\partial P)}{|P|^{2/3}},
\end{equation}
and ask which shape minimizes it.

The natural candidate is the regular truncated octahedron.  The corresponding
statement is usually referred to as the \emph{Truncated Octahedron Conjecture}
and is attributed to Bezdek; see \cite[Conjecture~1]{Langi2022}.  The conjecture
asserts that, among all three-dimensional parallelohedra of fixed volume, the
regular truncated octahedron has the smallest surface area.

This problem may be viewed as a lattice analogue of the classical
isoperimetric problem, and is also closely related to the Kelvin problem for
equal-volume partitions of space.  Other optimization problems on lattices,
such as sphere packing and quantization, likewise single out highly symmetric
lattices.  There is, however, no direct implication between these different
variational problems; see
\cite{BarnesSloane1983,CesaroniNovaga2024,CesaroniNovaga2025, gallagher}, for related
examples.

One of the main difficulties is the interaction between affine and Euclidean
geometry.  Three-dimensional parallelohedra admit a well-understood affine
classification, whereas the isoperimetric quotient \eqref{eq:isoperimetric-quotient}
is invariant under similarities but not under general affine transformations.
Thus an affine description of the possible combinatorial types does not by
itself identify the Euclidean minimizer.

Two particularly important examples are the Voronoi cells of the
body-centered cubic lattice ($\BCC$) and face-centered cubic lattice ($\FCC$).  The Voronoi cell of BCC
is the regular truncated octahedron, while that of FCC is the regular rhombic
dodecahedron.  Their isoperimetric quotients are
\[
 \Iso_\BCC
 =\frac{3(1+2\sqrt3)}{4^{2/3}}
 \approx 5.314739700,
 \qquad
 \Iso_\FCC
 =3\,2^{5/6}
 \approx 5.345392309.
\]
The   small gap between these two values shows how sharp is the 
 problem.

Our main result proves the Truncated Octahedron Conjecture and also identifies
the sharp minimizer among all the remaining Fedorov types.

\begin{theorem}[Truncated Octahedron Conjecture]
\label{thm:main}
For every three-dimensional parallelohedron $P$, one has
\begin{equation}\label{eq:main-full}
 \Iso(P)\ge
 \Iso_{\BCC}
 =\frac{3(1+2\sqrt3)}{4^{2/3}},
\end{equation}
with equality if and only if $P$ is similar to the regular truncated
octahedron.  In particular, for every three-dimensional lattice $\Lambda$,
\[
 \Iso(\Vor(\Lambda))\ge \Iso_{\BCC},
\]
with equality precisely for BCC, up to similarity.

If $P$ is not of truncated-octahedral Fedorov type, then the stronger estimate
\begin{equation}\label{eq:main-boundary}
 \Iso(P)\ge \Iso_{\FCC}=3\,2^{5/6}
\end{equation}
holds, with equality if and only if $P$ is similar to the regular rhombic
dodecahedron.  In particular, among lattice Voronoi cells with at most twelve
facets, equality occurs precisely at FCC.
\end{theorem}

Theorem~\ref{thm:main} does not address arbitrary periodic equal-volume
partitions, and hence does not solve the Kelvin problem.  It does, however,
settle the corresponding problem for convex lattice-translative tiles.  Indeed,
if a convex body $P\subset\mathbb R^3$ tiles space by translations along a
lattice, then $P$ is necessarily a parallelohedron; see, for instance,
\cite{McMullen1980}.

Let us briefly describe the main ideas of the proof.  We first consider lattice
Voronoi cells.  A classical result of Selling implies that every
three-dimensional lattice admits an obtuse superbase.  This gives a convenient
six-parameter description of its Voronoi cell and, through the geometry of
zonotopes, an explicit formula for its surface area.  The same description can
be encoded by the complete graph $K_4$, a representation that makes both the
volume and the facet structure particularly transparent.

The proof then separates naturally into the interior and the boundary of the
parameter space.  On the boundary we prove the sharp FCC inequality.  The
strata with few active generators are covered by the isoperimetric inequality
of Jo\'os and L\'angi \cite{JoosLangi2023}, while the remaining case is reduced
to a four-point determinant inequality proved here.

The interior case contains the BCC lattice and is the main part of the
argument.  The determinant expansion associated with the $K_4$ representation
defines a finite determinantal probability measure.  Its one-point marginals,
or leverage scores, allow us to reduce the original geometric inequality to a
small number of scalar constraints.  An entropy inequality, ultimately based
on the rank-one Ball--Barthe inequality \cite{Barthe1998}, provides the key
estimate.  The resulting scalar inequality is sharp, and equality forces all
six Selling parameters to coincide, which gives the BCC lattice.

The lattice result alone is not sufficient for the full conjecture, since an
arbitrary parallelohedron need not be a Voronoi cell in its given Euclidean
metric.  In the final step we therefore keep the same zonotopal description
but allow the ambient Euclidean metric to vary.  A determinant inequality
controls these additional affine degrees of freedom and shows that no affine
deformation can improve the BCC constant.  Combining this with the boundary
analysis yields Theorem~\ref{thm:main} for all three-dimensional
parallelohedra.

The approach builds on classical reduction theory.  Selling's reduction of
ternary quadratic forms \cite{Selling1874}, later incorporated into Voronoi's
general reduction theory \cite{Voronoi1908,Voronoi1909}, implies that every
three-dimensional lattice is of Voronoi's first kind.  Conway and Sloane gave
a modern formulation in terms of conorms and vonorms and recovered in this way
the five Fedorov types of three-dimensional lattice Voronoi cells
\cite{ConwaySloane1992article,ConwaySloane1999}.  We also use the classical
determinant formula for zonotope volumes \cite{Shephard1974}.  L\'angi proved
that the regular truncated octahedron minimizes mean width among
three-dimensional parallelohedra \cite{Langi2022}.

In our earlier work \cite{CesaroniNovaga2026}, we derived a closed formula for
the isoperimetric quotient in Selling coordinates, proved the strict local
 minimality of $\BCC$, and showed that $\FCC$ and the   cubic lattice $\text{SC}$ are not
local minimizers. This result has been recenlty improved by Lark Song, \cite{larksong} showing that the  truncated octahedron is a strict local minimizer of the isoperimetric quotient among all parallelohedra. 
The present paper establishes the corresponding global
result. 

The paper is organized as follows.  Section~3 introduces Selling coordinates,
the weighted $K_4$ zonotope, and the surface-area formula.  Section~4 proves
the sharp FCC boundary inequality, including the four-point determinant
inequality.  Section~5 treats the positive six-generator case and proves the
global BCC inequality for lattice Voronoi cells.  Section~6 extends the
argument to arbitrary three-dimensional parallelohedra and completes the proof
of the Truncated Octahedron Conjecture.  The final section discusses some open
problems.

We point out that ChatGPT (GPT-5.6 Sol) was used during the preparation of this work to double check algebraic identities and to assist in exploring proof arguments.

\smallskip

After this work was completed, we became aware of an independent proof of the Truncated Octahedron Conjecture by Thomas Hales and Lark Song \cite{HS26}.

\section{Notation}
  For $y=(y_1,\dots, y_d)\in\R^d$, the symbol $e_k(y)$, with $1\leq k\leq d$ 
denotes its $k$th elementary symmetric function, that is the sum of  all  products of $k$ distinct coordinates $y_i$ of $y$.   Let $\Omega$ be a finite probability space, given  $p$ and $q$ probability measures on $\Omega$, with the support of $p$ contained in the support of $q$, we denote  the relative entropy of $p$ with respect to $q$ (or  Kullback-Leibler divergence) as follows: 
\begin{equation}\label{KLdef}
 \KL(p\Vert q)=\sum_{\omega\in\Omega} p(\omega)
 \log\frac{p(\omega)}{q(\omega)},
\end{equation}
with the  convention $0\log(0)=0$. 

The Gram matrix $G$ associated with vectors $v_1,\ldots,v_m$ is the symmetric matrix with entries $g_{ij}=v_i\cdot v_j$.
The notation
$d^{\circ2}$ means entrywise squaring of a matrix $d$, that is, squaring each entry of the matrix, and
$\covol(\Lambda)$ is the Euclidean covolume of a lattice, that is, the volume of its fundamental parallelotope.  
We also let $K_4$ be the complete graph on four vertices.  A spanning tree
of a graph is a connected, acyclic subgraph containing all its vertices.
The group $S_4$
acts on Selling coordinates by relabelling the four vertices of the
superbase.

\section{The surface-area formula}\label{sec:geometry}
In this section we introduce and recall the main results about  lattice-Voronoi cells, and   we derive the exact surface-area formula.
The  Voronoi cell associated to a lattice $\Lambda$ in $\R^n$, denoted $\Vor(\Lambda)$,  is    an $n$-dimensional
convex polytope symmetric about the origin and  defined as 
\[\Vor(\Lambda):=\{x\in \R^n: |x|\le |x-\lambda|, \quad \forall \lambda\in\Lambda, \lambda \neq 0\}.\]
Equivalently $\Vor(\Lambda)$ can be defined as the intersection of the half spaces 
\[H_\lambda = \{x\in\R^n: 2 x\cdot \lambda\le |\lambda|^2\} \] for all $\lambda\in \Lambda\setminus \{0\}$.

\subsection{Selling-Voronoi reduction}

Lattices of Voronoi's first kind in $\R^n$   are  full-rank lattices which admit an obtuse
superbase, that  is $n+1$ vectors $v_0, \dots, v_n$ such that $v_1,\dots v_n$ is a lattice basis, $v_0+v_1+\dots +v_n=0$ and $v_i\cdot v_j\leq 0$ for $i\neq j$.  In dimension
three every lattice is of Voronoi's first kind, see the following result. 

\begin{proposition}
\label{thm:selling-reduction}
Every full-rank lattice $\Lambda\subset\mathbb R^3$ admits an obtuse
superbase $v_0,v_1,v_2,v_3$. The Selling parameters are defined as 
\[
 \rho_{ij}:=-v_i\cdot v_j\ge0\qquad(i\ne j).
\]   
Conversely, any six nonnegative numbers $\rho_{ij}$ whose associated Gram matrix 
is positive definite determine a  lattice in $\R^3$, up to orthogonal
isometry.  

\end{proposition}

\begin{proof}
The existence statement is Selling's reduction theorem,
also known as Voronoi's theorem (see \cite{Selling1874,Voronoi1908,Delone1937}).
For the definition of the superbase and of Selling parameters $\rho_{ij} $
see   \cite[Sections~2 and~7]{ConwaySloane1992article}.  If
$v_0=-v_1-v_2-v_3$, then
\[ \qquad v_i\cdot v_j=-\rho_{ij},\qquad
 v_i\cdot v_i=\rho_{0i}+\sum_{j\in\{1,2,3\}\setminus\{i\}}\rho_{ij}
\]
which gives the Gram matrix $A=(v_{i}\cdot v_j)_{1\leq i, j\leq 3}$ defined below.  From positive definiteness  of the matrix, one reconstructs the
lattice.   
\end{proof}

Let $K_4$ be the complete graph with vertex set $\{0,1,2,3\}$.
We may  identify the vertices with  the four
vectors $v_0,v_1,v_2,v_3$ of the superbase and  the Selling
parameter $\rho_{ij}$ with the weight of the edge $ij$ of $K_4$.  
Thus every full-rank lattice  is associated with a  nonnegative edge-weighted realization  of $K_4$.
We use throughout the fixed edge order
$
 (01,02,03,12,13,23)
$
and write
\[
\rho=(a,b,c,d,e,f)
   =(\rho_{01},\rho_{02},\rho_{03},
      \rho_{12},\rho_{13},\rho_{23})
   \in\mathbb R_{\ge0}^6.
\]
The associated   Gram matrix can be written as
\[
 A(\rho) =
 \begin{pmatrix}
 a+d+e&-d&-e\\
 -d&b+d+f&-f\\
 -e&-f&c+e+f
 \end{pmatrix},
 \qquad D(\rho)=\det A(\rho).
\]

A point $\rho$ of the Selling cone may be regarded as a nonnegative
weighting of the six edges of $K_4$, as shown in
Figure~\ref{fig:weighted-K4}.  Passing to the boundary amounts to
setting at least one of these edge weights equal to zero; the support
graph of the remaining positive edges will therefore provide a
convenient combinatorial description of the boundary strata considered
below.

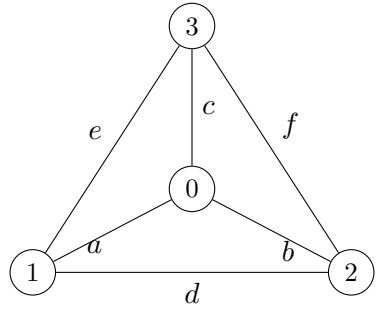
\begin{figure}[t]
\centering
\begin{tikzpicture}[scale=1.05,
    vertex/.style={circle,draw,fill=white,inner sep=1.6pt,
                   minimum size=6mm,font=\small}]

  \node[vertex] (0) at (0,1.05) {$0$};
  \node[vertex] (1) at (-2,0) {$1$};
  \node[vertex] (2) at (2,0) {$2$};
  \node[vertex] (3) at (0,3.1) {$3$};

  \draw (0)--node[below left] {$a$} (1);
  \draw (0)--node[below right] {$b$} (2);
  \draw (0)--node[right] {$c$} (3);

  \draw (1)--node[below] {$d$} (2);
  \draw (1)--node[above left] {$e$} (3);
  \draw (2)--node[above right] {$f$} (3);

\end{tikzpicture}
\caption{The weighted graph $K_4$ associated with the Selling vector
$\rho=(a,b,c,d,e,f)
=(\rho_{01},\rho_{02},\rho_{03},\rho_{12},\rho_{13},\rho_{23})$.}
\label{fig:weighted-K4}
\end{figure}

A reference realization of this abstract graph $K_4$ in $\R^3$ is the tetrahedron with vertices
\[
p_0=0,\qquad p_1=e_1,\qquad p_2=e_2,\qquad p_3=e_3.
\]
The corresponding edge vectors are
\[
 r_a=e_1,\quad r_b=e_2,\quad r_c=e_3,\quad
 r_d=e_1-e_2,\quad r_e=e_1-e_3,\quad r_f=e_2-e_3.
\]
Equivalently, the Gram matrix can be written as
\[A(\rho)= \sum_\xi  \xi r_\xi r_\xi^{\mathsf T}.\]

For a nontrivial subset $S\subset\{0,1,2,3\}$, let
\[
 c(S)=\sum_{ij\in\delta(S)}\rho_{ij}
\]
be the total weight of the cut $S\mid S^c$, where $\delta(S)$ denotes the
set of edges with one endpoint in $S$ and the other in $S^c$.  Modulo
complementation there are seven nontrivial cuts.  Their weights are:
\begin{align}
\notag
 c_0&=a+b+c,& c_1&=a+d+e,& c_2&=b+d+f,& c_3&=c+e+f,\\
 c_4&=b+c+d+e,& c_5&=a+c+d+f,& c_6&=a+b+e+f.\label{cutw}
\end{align}
Here $c_i=c(\{i\})$ for $i=0,1,2,3$, while $c_4,c_5,c_6$
correspond respectively to the cuts
$\{0,1\}\mid\{2,3\}$, $\{0,2\}\mid\{1,3\}$, and
$\{0,3\}\mid\{1,2\}$. 
For every cut $S\mid S^c$, we consider  the  spanning trees of the induced subgraphs $K_4[S]$ and $K_4[S^c]$, and we compute the  weighted spanning-tree polynomial  of $S\mid S^c$, i.e. the sum of the products of edge weights over all spanning trees.
So the following seven quadratic polynomials are exactly  the weighted spanning-tree polynomials associated to the previous seven cuts: 
\begin{align} \notag
 q_0&=de+df+ef,&q_1 &=bc+bf+cf,&q_2 &=ac+ae+ce,  & q_3 &=ab+ad+bd,
 \\ q_4& =af,&q_5 &=be,&q_6 &=cd.& \label{qdef}
\end{align}
The Gram determinant is
\begin{multline}
    \label{eq:D}
 D(\rho)=\det A(\rho) =abc+abe+abf+acd+acf+ade+adf+aef+
   \\ +bcd+bce+bde+bdf+bef+cde+cdf+cef.
\end{multline}
 The polynomial $D(\rho)$ is the weighted spanning-tree polynomial of $K_4$, that is the   sum of the products of edge weights over all spanning trees of $K_4$.
 
The weighted matrix--tree theorem
(see, for instance, \cite[Eq.~(1)]{KleeStamps2019})
gives the following identity:
\begin{equation}\label{eq:tree}
 \sum_{i=0}^6 q_i c_i=3D(\rho).
\end{equation}
Indeed, for each nontrivial cut $S\mid S^c$ of $K_4$, modulo
complementation, $q_i$ is the weighted spanning-tree polynomial of the
two induced subgraphs $K_4[S]$ and $K_4[S^c]$, while $c_i$ is the total
weight of the edges crossing the cut.  Hence $q_i c_i$ is the total
weight of spanning trees $T$ for which $S\mid S^c$ is obtained by
deleting one edge of $T$.  Conversely, deleting any edge of a spanning
tree produces exactly one such nontrivial cut.  Since every spanning
tree of $K_4$ has three edges, summing over the seven cuts counts every
weighted spanning tree exactly three times, which proves
\eqref{eq:tree}.

\subsection{Zonotopal representation of the Voronoi cells}

Using the geometric representation $r_\xi$ of the edges of the realization of $K_4$,   we define the centered graphical zonotope
\begin{equation}\label{eq:graphical-zonotope}
 Z_\rho=\sum_{\xi\in\{a,b,c,d,e,f\}}
 \left[-\frac{ \xi r_\xi}{2},\frac{ \xi r_\xi}{2}\right].
\end{equation}
Lattices of Voronoi's first kind are zonotopal lattices, that is, their
Voronoi cells are zonotopes; see, for instance,
\cite{McCormickPeisScheidweilerVallentin2021}.  In the present Selling
coordinates this structure takes the following explicit graphical form.

\begin{proposition}
\label{prop:voronoi-zonotope}
Let $\Lambda$ be a lattice in $\R^3$ with associated Selling parameters $\rho_{ij}$. If $D(\rho)>0$, then, up to an orthogonal change of coordinates, the Voronoi cell associated to $\Lambda$ is given by 
\begin{equation}\label{eq:voronoi-zonotope}
 \Vor(\Lambda)=A(\rho)^{-1/2}Z_\rho.
\end{equation}
Moreover $|Z_\rho|=D(\rho)$.
\end{proposition}
\begin{proof}
Choose the lattice basis with Gram matrix $A(\rho)$, so that
$\Lambda=A^{1/2}(\rho)\mathbb Z^3$.  For $n\in\mathbb Z^3$, the support function of
$Z_\rho$ is
\[
 h_{Z_\rho}(n)=\frac12\sum_\xi  \xi|n\cdot r_\xi|.
\]
Since $n\cdot r_\xi\in \mathbb Z$, one has
$|n\cdot r_\xi|\le(n\cdot r_\xi)^2$.  Recalling that 
$A(\rho)=\sum_\xi  \xi r_\xi r_\xi^{\mathsf T}$,
\[
 h_{Z_\rho}(n)\le\frac12 nA(\rho)n^{\mathsf T}.
\]
These are precisely the Voronoi half-space
inequalities for the lattice $A^{1/2}(\rho)\mathbb Z^3$ applied to vectors scaled by  $A^{-1/2}(\rho)$.   Thus
$A^{-1/2}(\rho)Z_\rho\subseteq\Vor(\Lambda)$.

By Shephard's zonotope decomposition formula \cite[Eq. (57)]{Shephard1974},
\[
 |Z_\rho|
 =\sum_{\substack{I\subset E(K_4)\\ |I|=3}}
   \left(\prod_{\xi\in I}\rho_\xi\right)
   \left|\det(r_\xi)_{\xi\in I}\right|,
\]
where $\rho_\xi$ denotes the weight of the edge $\xi$.  For the incidence
roots of $K_4$, the determinant has absolute value one precisely when $I$ is
a spanning tree, and vanishes otherwise.  Hence
\[
 |Z_\rho|
 =\sum_{T\in\mathcal T(K_4)}\prod_{\xi\in T}\rho_\xi
 =D(\rho),
\]
and the image of $Z_\rho$ under $A^{-1/2}(\rho)$ has volume
$D^{1/2}(\rho)=|\Vor(\Lambda)|$.  
Recalling that $A^{-1/2}Z_\rho\subseteq\Vor(\Lambda)$, this implies the conclusion.
 \end{proof}

Finally we recall  that  the strict inequalities $\rho_{ij}>0$ for all $i\neq j$ correspond to  a lattice whose Voronoi cell has  
fourteen facets, whereas the class  of full-rank lattices with at least one vanishing Selling parameter corresponds to the nondegenerate boundary of the Selling cone  and the  Voronoi cells associated to such lattices have at most twelve facets.

Given the nonnegative edge-weighted
realization of $K_4$ associated to a lattice with Selling parameters $\rho_{ij}$, the active graph $G_\rho$ consists of the subgraph of $K_4$ with the same number of vertices and  only with the  edges which have  positive weight.

We have the following result (for the proof we refer to \cite{ConwaySloane1992article}). 

\begin{proposition}
\label{prop:zonotope-facets}
If $D(\rho)>0$, the active graph $G_\rho\subset K_4$ is connected.  The opposite
facet pairs of $\Vor(\Lambda)$ are in bijection with unordered nontrivial bipartitions
$S\mid S^c$ for which both induced graphs $G_\rho[S]$ and $G_\rho[S^c]$ are
connected.  Consequently, $\Vor(\Lambda)$ has fourteen facets if and only if all six
Selling parameters are positive; otherwise $\Vor(\Lambda)$ has at most
twelve facets.
\end{proposition}
\begin{corollary} 
\label{cor:twelve-facets-boundary}
Let $\Lambda\subset\R^3$ be a full-rank lattice and $P=\Vor(\Lambda)$.  For
any non-negative Selling parameter vector obtained from an obtuse superbase,
the following are equivalent:
\begin{enumerate}
 \item $P$ has at most twelve facets;
 \item at least one Selling parameter vanishes;
 \item the parameter vector lies on the nondegenerate boundary
 \[
 \left\{\rho\in\mathbb R_{\ge0}^6:D(\rho)>0,\ \min_{ij}\rho_{ij}=0\right\}.
 \]
\end{enumerate}
Moreover, $P$ has fourteen facets if and only if all six parameters are
positive; in this case $P$ is combinatorially a truncated octahedron.
\end{corollary}


 
\subsection{The explicit formula for the isoperimetric quotient} \label{esplicita}

We associate to the seven nontrivial vertex cuts of $K_4$, modulo complementation, the following vectors  
\begin{align}\notag
& z_0=(1,1,1)^\trans,
& z_1=(1,0,0)^\trans,\qquad
&z_2=(0,1,0)^\trans,
&z_3=(0,0,1)^\trans,\\
 &z_4=(0,1,1)^\trans,
&z_5=(1,0,1)^\trans,\qquad 
&z_6=(1,1,0)^\trans.&\label{zi}
\end{align}
Indeed the facet of $Z_\rho$ exposed by $z_i$ is, up to translation, the planar zonotope
generated by the active edges internal to the two sides of the cut.

\begin{proposition}\label{prop:surface-formula}
 For every lattice $\Lambda$ with   $D(\rho)>0$, the scale-invariant
isoperimetric quotient of the lattice Voronoi cell is
\begin{equation}\label{eq:closed-F}
\Iso(\Vor(\Lambda))=F(\rho)=\frac{2N(\rho)}{D(\rho)^{5/6}},\qquad
 N(\rho)=\sum_{i=0}^6q_i\sqrt{c_i}.
\end{equation}
The total surface area is $ \mathcal H^2(\partial\Vor(\Lambda))=2N(\rho)/\sqrt{D(\rho)}$.
\end{proposition}

\noindent This formula has been derived in \cite[Section 2]{CesaroniNovaga2026}, for the sake of completeness we provide the proof also here. 

We recall that \begin{align*}
\Iso_\BCC&=F(1,1,1,1,1,1) =\frac{3(1+2\sqrt3)}{4^{2/3}}:=F_\BCC\\\Iso_\FCC&=F(0,1,1,1,1,0) =3\ 2^{\frac56}:=F_\FCC.\end{align*}
\begin{proof}
  For planar generators
$g_1,\ldots,g_m$ lying in a plane with unit normal $\nu$, the 
two-dimensional zonotope formula is
\[
 \mathcal H^2\left(\sum_{k=1}^m[0,g_k]\right)
 =\sum_{1\le r<s\le m}|(g_r\times g_s)\cdot\nu|.
\]
We are going to apply this  formula to $g_\xi= \xi r_\xi$ with $\xi\in\{a,b,c,d,e,f\}$ and
$\nu=z_i/|z_i|$, where $z_i$ are defined in \eqref{zi}. 

More precisely, for a singleton cut, say $\{0\}\mid\{1,2,3\}$,
the free edges are $12,13,23$.  Their three cross products are parallel to
$z_0$ and have respective magnitudes $de|z_0|$, $df|z_0|$ and
$ef|z_0|$; their sum is $q_0|z_0|$.  Relabelling gives the other singleton
cuts.  

For a $2+2$ cut there is one internal edge on each side; for example
$\{0,1\}\mid\{2,3\}$ gives the parallelogram generated by
$ar_a$ and $fr_f$, whose area is $af|z_4|=q_4|z_4|$.  

Thus in every case
\begin{equation}\label{eq:facet-Z-area}
 \mathcal H^2(F_i(Z_\rho))=q_i|z_i|.
\end{equation}
Now we recall by Proposition \ref{prop:voronoi-zonotope} that $
\Vor(\Lambda)=A(\rho)^{-1/2}Z_\rho$. 

If an invertible linear map $L$ acts on a planar set with normal $n$, its area
is multiplied by $|\det L|\,|L^{-\mathsf T}n|/|n|$.  Apply this with
$L=A^{-1/2}(\rho)$ and $n=z_i$, using the fact that   $z_i^{\mathsf T}A(\rho)z_i=c_i$, we obtain from
\eqref{eq:facet-Z-area}  that 
\[
 \mathcal H^2(F_i(\Vor(\Lambda)))
 =D^{-1/2}(\rho)\frac{\sqrt{c_i}}{|z_i|}\,q_i|z_i|
 =\frac{q_i\sqrt{c_i}}{\sqrt{  D(\rho)}}.
\]
The opposite facet has the same area.  Summing the seven opposite pairs we obtain the total surface area of the Voronoi cell \[
 \mathcal H^2(\partial\Vor(\Lambda))= 2\sum_{i=0}^6 \frac{q_i\sqrt{c_i}}{\sqrt {D(\rho)}}
.\] To conclude we divide it 
  by $D^{1/3}(\rho)$.
\end{proof}

\smallskip

 We conclude with an observation that will be useful when proving the interior isoperimetric inequality  (that is, among lattices with Selling parameters  $\rho_{ij}>0$ for all $i, j$).  
Let
\begin{equation}\label{eq:M-J}
 M(\rho)=\sum_{i=0}^6\frac{q_i}{\sqrt{c_i}}z_i z_i^\trans,
 \qquad
 \cJ(\rho)=\frac{\det M(\rho)}{D^{3/2}(\rho)}.
\end{equation}

In the positive cone of $\rho_{ij}> 0$, both $A(\rho)$ and $M(\rho)$ are positive definite.  Since
$z_i^{\mathsf T}A(\rho)z_i=c_i$,
\[
 \operatorname{tr}(A^{1/2}(\rho)M(\rho)A^{1/2}(\rho))
 =\sum_{i=0}^6\frac{q_i}{\sqrt{c_i}}\,z_i^{\mathsf T}A(\rho)z_i
 =\sum_{i=0}^6q_i\sqrt{c_i}=N(\rho).
\]
Let $\lambda_1,\lambda_2,\lambda_3>0$ be the eigenvalues of
$A^{1/2}(\rho)M(\rho)A^{1/2}(\rho)$.  Then
\[
 N(\rho)=\sum_{j=1}^3\lambda_j
 \ge3(\lambda_1\lambda_2\lambda_3)^{1/3}
 =3(D(\rho)\det M(\rho))^{1/3},
\]
by the arithmetic--geometric mean inequality.  Therefore, using
\eqref{eq:closed-F} and \eqref{eq:M-J}, we obtain the following inequality relating the isoperimetric quotient with the determinant $\cJ(\rho)$
\begin{equation}\label{eq:spectral}
 F(\rho)=\frac{2N(\rho)}{D(\rho)^{5/6}}
 \ge6\left(\frac{\det M(\rho)}{D(\rho)^{3/2}}\right)^{1/3}
 =6\cJ^{1/3}(\rho).
\end{equation}
The  BCC lattice has Selling parameters $\rho_{ij}\equiv 1$, that is $a=\cdots=f=1$. 
Therefore, a direct substitution gives
\[
 A^{1/2}(1)M(1)A^{1/2}(1)=(2+4\sqrt3)I_3
\] where $I_3$ is the identity matrix. 
For every BCC rescaling, $A^{1/2}(\rho)M(\rho)A^{1/2}(\rho)$ remains a positive scalar multiple of the
identity.  Thus the three eigenvalues $\lambda_i$ coincide,
and equality is attained in \eqref{eq:spectral}. 
Moreover,
\begin{equation}\label{eq:constants}
 F_{\BCC}=F(1)=\frac{3(1+2\sqrt3)}{4^{2/3}},
 \qquad
  \cJ_{\BCC}= \cJ(1)=\left(\frac{F_{\BCC}}6\right)^3
 =\frac{37+30\sqrt3}{128}.
\end{equation}
Therefore the inequality
\begin{equation}\label{eq:sharpdet}
\cJ(\rho)\ge \cJ_{\BCC}
\end{equation} 
implies the isoperimetric inequality $F(\rho)\ge F_{\BCC}$.

\section{The FCC theorem}\label{sec:FCC}
This section  is devoted to the analysis of the isoperimetric inequality for  full-rank lattices with at least one vanishing Selling parameter. 
In particular, we prove the
sharp boundary theorem for lattice Voronoi cells, showing that the unique minimizer in this case is the Voronoi cell of the $\FCC$ lattice.

 First of all, we recall a known result by  Jo\'os and L\'angi
\cite{JoosLangi2023} showing that 
among  lattices with at least two vanishing Selling parameters,  the $\FCC$ lattice has the Voronoi cell with minimal isoperimetric quotient. 
\begin{proposition} \label{prop:four-generators}
If $D(\rho)>0$ and at most four Selling parameters are positive, then
\[
 F(\rho)\ge F_{\FCC}=3\,2^{5/6}.
\]
Equality holds precisely on the $S_4$-orbit of the $\FCC$ lattice, which has Selling parameters $(0,m,m,m,m,0)$, for $m>0$.
\end{proposition}

\begin{proof}
By Proposition~\ref{prop:voronoi-zonotope}, $\Vor(\Lambda)$ is a full-dimensional
zonotope generated by three or four nonzero segments.  
Jo\'os and L\'angi
\cite[Theorem~3, with $d=3$ and $k=2$]{JoosLangi2023}
prove that, among three-dimensional zonotopes generated by three or four
segments and of fixed volume, the second intrinsic volume, that in dimension $3$ is given by $\mathcal H^2(\partial \Vor(\Lambda))/2$  is uniquely
minimized, respectively, by the cube $\mathrm{SC}$ which is associated to Selling parameters $(1,1,1,0,0,0)$ and by the regular rhombic
dodecahedron $\FCC$. Thus their result gives the sharp surface-area bounds
\[
 F(\rho)\geq
\min\{F_{\mathrm{SC}},F_{\FCC}\}
 =F_{\FCC}=3\,2^{5/6},
\]
where $F_{\mathrm{SC}}=F(1,1,1,0,0,0)=6$.  It remains to identify equality in Selling
coordinates.  Since $F_{\mathrm{SC}}>F_{\FCC}$, equality requires exactly
four positive coordinates and the corresponding zonotope must be a regular
rhombic dodecahedron.  A connected four-edge graph on four vertices is either
a $4$-cycle or a triangle with a pendant edge.  The latter is impossible,
because the three roots belonging to the triangle are linearly dependent,
whereas every three generators of a regular rhombic dodecahedron are
independent.

On a $4$-cycle, orient the four roots cyclically.  Their unique linear
relation has coefficients of equal absolute value.  The physical generators
are
\[
 g_\xi=\rho_\xi A(\rho)^{-1/2}r_\xi.
\]
Consequently, the coefficients in their unique dependence are proportional
to $\rho_\xi^{-1}$.  For a regular rhombic dodecahedron, after a suitable
orientation, the four generators are the vertices of a regular tetrahedron
centred at the origin, so the coefficients in their unique dependence have
equal absolute value.  Hence the four active Selling parameters are equal.
The three $4$-cycles of $K_4$ are the complements of the three pairs of
opposite edges, which is exactly the $S_4$-orbit of
$(0,m,m,m,m,0)$, $m>0$.
\end{proof}

We are therefore left to prove the same result for lattices with exactly one
vanishing Selling parameter.  In this case the result of Jo\'os and L\'angi
in \cite{JoosLangi2023} does not apply.  Although a Voronoi cell with five
positive Selling parameters is a five-generator zonotope, the class of all
five-generator zonotopes is strictly larger: some members of that larger
class have a smaller isoperimetric quotient than the regular rhombic
dodecahedron, but they do not satisfy the lattice--Voronoi metric relation.

In order to prove the result,
we need to use a different representation  of the isoperimetric quotient $F$.    

\subsection{Lattices with at least one null Selling parameter}
By tetrahedral symmetry it suffices to take $a=0$ and let $f=t$, be a varying parameter (that could also be $0$). In this case the weighted spanning-tree polynomials $q_i$ defined in \eqref{qdef} and the weights of the cuts defined in \eqref{cutw} are 
\begin{align*}
&q_0= de+t(d+e), &&c_0= b+c,\\
&q_1= bc+t(b+c), &&c_1= d+e,\\
&q_2= ce, &&c_2= b+d+t,\\
&q_3= bd, &&c_3= c+e+t,\\
&q_4=0  &&c_4= b+c+d+e,\\
&
q_5= be &&c_5= c+d+t  \\
& q_6= cd. &&c_6= b+e+t
\end{align*}
We recall from  \eqref{eq:closed-F}   and \eqref{eq:D} that $N(\rho)=\sum_{i=0}^6 q_i\sqrt{c_i}$
\[
D(\rho)=bcd+bce+bde+bdt+bet+cde+cdt+cet.
\]
The term containing $\sqrt{c_4}$ vanishes because $q_4=0$, so it will be
omitted below.

We define 
\begin{equation}\label{metric}
 d_{01}=\sqrt{c_0},\ d_{23}=\sqrt{c_1},\ d_{02}=\sqrt{c_2},
 \ d_{13}=\sqrt{c_3},\ d_{12}=\sqrt{c_5},\ d_{03}=\sqrt{c_6}
\end{equation}
and we observe that 
the array $d$ is an element of the metric cone on four points according to this definition.  
\begin{definition}[Metric cone on four points]
Let us fix 
  four points $0,1,2,3$ and define $\mathrm{MET}_4$ as the set of all possible semimetrics on $\{0,1,2,3\}.$ More precisely    $d=(d_{ij})_{0\le i,j\le3}\in \mathrm{MET}_4$ if and only if   $d_{ii}=0$, $d_{ij}=d_{ji}\geq 0$, and  $d_{ik}\leq d_{ij}+d_{jk}$.
$\mathrm{MET}_4$  is a cone, that is, it is closed by addition and multiplication by positive scalar.
\end{definition} 
 \begin{lemma} Let $d$ be defined as in \eqref{metric} and let us  extend it to an array $d_{ij}$, for all $0\leq i, j\leq 3$ by imposing   $d_{ij}=d_{ji}$ for $i>j$ and $d_{ii}=0$.  Then $d\in \mathrm{MET}_4$.   
 \end{lemma}\begin{proof}
  It is sufficient to check the triangle inequality, by direct computation.  Indeed,  for the four sets of indices  
$ijk= 012,013,023,123$, the unordered triples of quantities
$d_{ij}^2+d_{ik}^2-d_{jk}^2$ are respectively
\[
 2(b,c,d+t),\qquad 2(b,c,e+t),\qquad
 2(d,e,b+t),\qquad 2(d,e,c+t),
\]
and are nonnegative.  Hence,
$d_{ij}^2\le d_{ik}^2+d_{jk}^2$.  Taking square roots  gives
$d_{ij}\le d_{ik}+d_{jk}$; thus the positive square roots satisfy all
triangle inequalities. 
 \end{proof}
For $d\in \mathrm{MET}_4$, we let
\[
 \Gamma(d)=\begin{pmatrix}
 d_{01}&(d_{01}+d_{02}-d_{12})/2&(d_{01}+d_{03}-d_{13})/2\\
 (d_{01}+d_{02}-d_{12})/2&d_{02}&(d_{02}+d_{03}-d_{23})/2\\
(d_{01}+d_{03}-d_{13})/2&(d_{02}+d_{03}-d_{23})/2&d_{03}
 \end{pmatrix}
\]
and 
\[\Gamma(d^{\circ2})=\begin{pmatrix}
 d_{01}^2&(d_{01}^2+d_{02}^2-d_{12}^2)/2&(d_{01}^2+d_{03}^2-d_{13}^2)/2\\
 (d_{01}^2+d_{02}^2-d_{12}^2)/2&d_{02}^2&(d_{02}^2+d_{03}^2-d_{23}^2)/2\\
(d_{01}^2+d_{03}^2-d_{13}^2)/2&(d_{02}^2+d_{03}^2-d_{23}^2)/2&d_{03}^2
 \end{pmatrix}.
\]

We will denote $\mathsf G=\Gamma(d^{\circ2})$, $\mathsf H=\Gamma(d)$, where $d$ is as in \eqref{metric}.

Substituting the $c_i$ in the definition of $\Gamma(d )$, 
$\Gamma(d^{\circ2})$ gives the concrete matrices 
\[
 \mathsf H=
 \begin{pmatrix}
 \sqrt{c_0}&\frac{\sqrt{c_0}+\sqrt{c_2}-\sqrt{c_5}}{2}&\frac{\sqrt{c_0}+\sqrt{c_6}-\sqrt{c_3}}{2}\\
\frac{\sqrt{c_0}+\sqrt{c_2}-\sqrt{c_5}}{2}&\sqrt{c_2}&\frac{\sqrt{c_2}+\sqrt{c_6}-\sqrt{c_1}}{2}\\
 \frac{\sqrt{c_0}+\sqrt{c_6}-\sqrt{c_3}}{2}&\frac{\sqrt{c_2}+\sqrt{c_6}- \sqrt{c_1}}{2}&\sqrt{c_6}
 \end{pmatrix} 
\] and 
\[
 \mathsf G=
 \begin{pmatrix}
 c_0&\frac{c_0+c_2-c_5}{2}&\frac{c_0+c_6-c_3}{2}\\
\frac{c_0+c_2-c_5}{2}&c_2&\frac{c_2+c_6-c_1}{2}\\
 \frac{c_0+c_6-c_3}{2}&\frac{c_2+c_6-c_1}{2}&c_6
 \end{pmatrix}=
 \begin{pmatrix}
 b+c&b&b\\
 b&b+d+t&b+t\\
 b&b+t&b+e+t
 \end{pmatrix}.
\]
Notice that $\mathsf G$ has the positive-semidefinite rank-one decomposition
\[
 \mathsf G=b(1,1,1)^{\mathsf T}(1,1,1)+c e_1e_1^{\mathsf T}+d e_2e_2^{\mathsf T}
 +e e_3e_3^{\mathsf T}
 +t(0,1,1)^{\mathsf T}(0,1,1).
\]
Computing the determinant of $\mathsf G$ we have
\[
\begin{aligned}
\det\mathsf G=&\frac14\Big(
-c_0^2c_1-c_0c_1^2
+c_0c_1c_2+c_0c_1c_3+c_0c_1c_5+c_0c_1c_6+c_0c_2c_3\\
&\quad\, -c_0c_2c_5-c_0c_3c_6+c_0c_5c_6
+c_1c_2c_3-c_1c_2c_6-c_1c_3c_5+c_1c_5c_6\\
&\quad\, -c_2^2c_3-c_2c_3^2+c_2c_3c_5+c_2c_3c_6
+c_2c_5c_6+c_3c_5c_6-c_5^2c_6-c_5c_6^2
\Big)\\
=&\, bcd+bce+bde+bdt+bet+cde+cdt+cet\\
=&\, D(\rho).
\end{aligned}
\]
Direct differentiation gives
\[\frac{\partial \det\mathsf G}{\partial c_i}= q_i\] 
for all $0\le i\le6$, with $i\neq 4$.  Note also that there holds 
\[
 \mathsf H=\Gamma(d)=\sum_{i\in 
\{0,1,2,3,5,6\}} \sqrt{c_i}\,
 \frac{\partial\mathsf G}{\partial c_i}.
\]
Jacobi's formula, in its polynomial form gives for every $i\neq 4$
\[ \sqrt{c_i}\,\frac{\partial \det\mathsf G}{\partial c_i} = \sqrt{c_i}\tr\left(\adj \mathsf G \ \frac{\partial\mathsf G}{\partial c_i}\right)=\tr\left(\adj \mathsf G \  \sqrt{c_i}\frac{\partial\mathsf G}{\partial c_i}\right)\] where $\adj \mathsf G$ is the transpose of the cofactor matrix of $ \mathsf G$, that is $(\det\mathsf G) \mathsf G^{-1}$.
Therefore we get 
\begin{equation}\label{eq:facet-structure}
 \det\mathsf G=D(\rho),\qquad
  N(\rho)= \sum_{i=0}^6q_i\sqrt{c_i}
 =\tr(\adj\mathsf G\,\mathsf H)=\det\mathsf G \tr( \mathsf G^{-1/2}\,\mathsf H\mathsf G^{-1/2}).
\end{equation}
Recalling \eqref{eq:closed-F}, the isoperimetric quotient can be expressed as 
\begin{equation}\label{nuovaF}F(\rho)= 2 \det\mathsf G ^{1/6} \tr( \mathsf G^{-1/2}\,\mathsf H\mathsf G^{-1/2}).\end{equation}
\subsection{A determinant inequality}
We provide now a general result relating the determinant of matrices $\Gamma(d)$ and $\Gamma(d^{\circ2})$ for $d\in \mathrm{MET}_4$.

First of all we introduce the notion of cut cone. 
We call a nonempty proper subset of $\{0,1,2,3\}$ a cut and identify a cut
with its complement.  Thus there are four singleton cuts and three $2+2$
cuts.  

\begin{definition}[cut cone]
The cut metric associated to a cut $S$ is  the semimetric defined as 
\[\delta_S(i,j):=\begin{cases}0 & i, j\in S\\ 0 & i,j\in S^c=\{0,1,2,3\}\setminus S\\ 1&\text{elsewhere}. \end{cases} \]
The cut cone is the conic hull
\[
 \mathrm{CUT}_4
 =\operatorname{cone}\{\delta_S:
 S\subsetneq\{0,1,2,3\},\, S\ne \emptyset\}
 \subset\mathrm{MET}_4.
\]
\end{definition} 
 
We recall the following result (see \cite[Chapter~4]{DezaLaurent1997}).
\begin{proposition}
There holds 
\[\mathrm{MET}_4=\mathrm{CUT}_4.\] 
In particular, any semimetric in the metric cone is a nonnegative
combination of the  seven cut metrics.  
\end{proposition}
Notice that the cut metrics satisfy the identity
\[
 \delta_{\{0\}}+\delta_{\{1\}}+\delta_{\{2\}}+\delta_{\{3\}}
 \equiv\delta_{\{0,1\}}+\delta_{\{0,2\}}+\delta_{\{0,3\}}.
\]
Starting from any nonnegative cut decomposition, subtract the smallest of the
three $2+2$ coefficients from all three of them and add the same quantity to
each singleton coefficient.  The displayed identity shows that the metric is
unchanged, while one $2+2$ coefficient becomes zero.  

We shall also use that the two determinants appearing below,
$\det\Gamma(d)$ and $\det\Gamma(d^{\circ2})$, are invariant
under relabelling of the four points.

After a permutation of the four points, we may write
for every $d\in \mathrm{MET}_4 $ 
\begin{equation}\label{eqrappr}
 d=\sum_{i=0}^3x_i\delta_{\{i\}}+u\delta_{\{0,1\}}+v\delta_{\{0,2\}},
 \qquad x_i,u,v\ge0. 
\end{equation}
 
In these coordinates the six distances are
\begin{align*}
 d_{01}&=x_0+x_1+v, & d_{02}&=x_0+x_2+u,\\
 d_{03}&=x_0+x_3+u+v, & d_{12}&=x_1+x_2+u+v,\\
 d_{13}&=x_1+x_3+u, & d_{23}&=x_2+x_3+v.
\end{align*}

We have the following theorem. 
\begin{theorem}\label{thm:fourpoint}
For every semimetric $d\in\mathrm{MET}_4$,
\begin{equation}\label{eq:fourpoint}
 2\big(\det\Gamma(d)\big)^2\ge\det\Gamma(d^{\circ2}),
\end{equation}
where $(d^{\circ2})_{ij}=d^2_{ij}$.
If $\det\Gamma(d)>0$, equality holds if and only if the six distances are
equal.
\end{theorem}

\begin{proof}
Let $x_0,\ldots,x_3,u,v$ be as in \eqref{eqrappr}. Set
\[
s:=x_0+x_1+x_2+x_3,
\qquad
\pi:=x_0x_1x_2x_3,
\qquad
M_i:=\prod_{j\ne i}x_j,
\quad 0\le i\le3.
\]
For a cut $S\subset\{0,1,2,3\}$, set
\[
r_S:=
\bigl(
\delta_S(0,1),
\delta_S(0,2),
\delta_S(0,3)
\bigr)^{\mathsf T}.
\]
Since
\[
\delta_S(i,j)
=
\delta_S(0,i)+\delta_S(0,j)
-2\delta_S(0,i)\delta_S(0,j),
\]
the definition of $\Gamma(\delta_S)$ gives
\[
\Gamma(\delta_S)=r_Sr_S^{\mathsf T}.
\]
For the six cuts appearing in \eqref{eqrappr}, the corresponding vectors are
\[
\begin{aligned}
r_{\{0\}}&=(1,1,1)^{\mathsf T},&
r_{\{1\}}&=(1,0,0)^{\mathsf T},&
r_{\{2\}}&=(0,1,0)^{\mathsf T},\\
r_{\{3\}}&=(0,0,1)^{\mathsf T},&
r_{\{0,1\}}&=(0,1,1)^{\mathsf T},&
r_{\{0,2\}}&=(1,0,1)^{\mathsf T}.
\end{aligned}
\]
Hence, by linearity of $\Gamma$,
\[
\Gamma(d)
=
\sum_{i=0}^3 x_i r_{\{i\}}r_{\{i\}}^{\mathsf T}
+u\,r_{\{0,1\}}r_{\{0,1\}}^{\mathsf T}
+v\,r_{\{0,2\}}r_{\{0,2\}}^{\mathsf T}.
\]
Thus $\Gamma(d)=BB^{\mathsf T}$, where
\[
B:=
\begin{pmatrix}
\sqrt{x_0}\,r_{\{0\}}&
\sqrt{x_1}\,r_{\{1\}}&
\sqrt{x_2}\,r_{\{2\}}&
\sqrt{x_3}\,r_{\{3\}}&
\sqrt{u}\,r_{\{0,1\}}&
\sqrt{v}\,r_{\{0,2\}}
\end{pmatrix}.
\]
and
\[
\det\Gamma(d)
=
\sum_{\substack{I\subset\{1,\ldots,6\}\\ |I|=3}}
(\det B_I)^2.
\]
The nonzero $3\times3$ minors of the unweighted column configuration
have determinant of absolute value one. Among the minors containing
$r_{\{0,1\}}$, but not $r_{\{0,2\}}$, the nonzero ones correspond to
the products
\[
x_0x_2,\quad x_0x_3,\quad x_1x_2,\quad x_1x_3,
\]
while those containing $r_{\{0,2\}}$, but not $r_{\{0,1\}}$, correspond to
\[
x_0x_1,\quad x_0x_3,\quad x_1x_2,\quad x_2x_3.
\]
All four minors containing both of the last two columns are nonzero.
Thus Cauchy--Binet formula gives
\begin{equation}\label{eq:fourpoint-P}
\begin{aligned}
P:=\det\Gamma(d)
={}&e_3(x)
+u(x_0+x_1)(x_2+x_3)\\
&+v(x_0+x_2)(x_1+x_3)
+uv(x_0+x_1+x_2+x_3)\\
={}&e_3(x)
+u(x_0+x_1)(x_2+x_3)
+v(x_0+x_2)(x_1+x_3)
+uvs.
\end{aligned}
\end{equation}

Set now
\[
R:=2P^2-\det\Gamma(d^{\circ2})
\]
and define
\begin{align*}
U&:=((x_0+x_1)(x_2+x_3))^2-8\pi,\\
V&:=((x_0+x_2)(x_1+x_3))^2-8\pi,\\
W&:=(x_0x_3+x_1x_2)^2-8\pi,
\end{align*}
together with
\begin{align*}
\Lambda_u:=4[&
x_0x_1(x_0+x_1)(x_2-x_3)^2
+x_2x_3(x_2+x_3)(x_0-x_1)^2],\\
\Lambda_v:=4[&
x_0x_2(x_0+x_2)(x_1-x_3)^2
+x_1x_3(x_1+x_3)(x_0-x_2)^2],\\
B_{21}:=4[&
x_0x_2(x_0+x_2)+x_0x_3(x_0+x_3)
+x_1x_2(x_1+x_2)+x_1x_3(x_1+x_3)],\\
B_{12}:=4[&
x_0x_1(x_0+x_1)+x_0x_3(x_0+x_3)
+x_1x_2(x_1+x_2)+x_2x_3(x_2+x_3)].
\end{align*}

Using
\[
\det
\begin{pmatrix}
a&p&q\\
p&b&r\\
q&r&c
\end{pmatrix}
=
abc+2pqr-ar^2-bq^2-cp^2
\]
and collecting the terms with the same powers of $u$ and $v$, we obtain
\begin{align}\label{eq:fourpoint-residual}
R={}&
2\sum_{i<j}(M_i-M_j)^2
+\Lambda_u u+\Lambda_v v
+2(Uu^2+2Wuv+Vv^2)\notag\\
&+B_{21}u^2v+B_{12}uv^2
+2u^2v^2
\bigl(2(u+v)^2+4s(u+v)+3s^2\bigr).
\end{align}

All the terms on the right-hand side are manifestly nonnegative,
except a priori the quadratic form
$Uu^2+2Wuv+Vv^2.$
Indeed,
\[
(x_0+x_1)(x_2+x_3)
\ge
4(x_0x_1x_2x_3)^{1/2}
=4\sqrt{\pi},
\]
and similarly
\[
(x_0+x_2)(x_1+x_3)\ge4\sqrt{\pi}.
\]
Hence
\[
U\ge8\pi,
\qquad
V\ge8\pi.
\]
Moreover,
\[
x_0x_3+x_1x_2
\ge2\sqrt{x_0x_1x_2x_3}
=2\sqrt{\pi},
\]
and therefore
\[
W\ge-4\pi.
\]
Since $u,v\ge0$, it follows that
\[
\begin{aligned}
Uu^2+2Wuv+Vv^2
&\ge
8\pi u^2-8\pi uv+8\pi v^2\\
&=
8\pi(u^2-uv+v^2)
\ge0.
\end{aligned}
\]
Thus every term in \eqref{eq:fourpoint-residual} is nonnegative, and hence
\[
2\bigl(\det\Gamma(d)\bigr)^2
\ge
\det\Gamma(d^{\circ2}).
\]

\smallskip

We now discuss the equality case. Assume $P=\det\Gamma(d)>0$ and $R=0$.
The last term in \eqref{eq:fourpoint-residual} forces
\[
uv=0.
\]
Suppose, for instance, that $u>0$ and $v=0$. Since all terms in
\eqref{eq:fourpoint-residual} are nonnegative, equality implies $U=0$.
But
\[
U
=
(x_0^2+x_1^2)(x_2+x_3)^2
+2x_0x_1(x_2-x_3)^2.
\]
Thus either $x_0=x_1=0$ or $x_2=x_3=0$. In either case
\eqref{eq:fourpoint-P} gives $P=0$, a contradiction.
The case $v>0$ and $u=0$ is identical, using
\[
V
=
(x_0^2+x_2^2)(x_1+x_3)^2
+2x_0x_2(x_1-x_3)^2.
\]
Therefore
\[
u=v=0.
\]

Equality in \eqref{eq:fourpoint-residual} then yields
\[
M_0=M_1=M_2=M_3.
\]
Since $u=v=0$,
\[
P=e_3(x)=M_0+M_1+M_2+M_3>0.
\]
Hence their common value is $P/4>0$, so in particular every $M_i$ is
positive. It follows that
\[
x_0,x_1,x_2,x_3>0.
\]
Writing
\[
M_i=\frac{\pi}{x_i},
\]
the equality of the four $M_i$ implies
\[
x_0=x_1=x_2=x_3.
\]
Since $u=v=0$, the six distances in \eqref{eqrappr} are therefore equal.

Conversely, if all six distances are equal to $m>0$, then, writing
$\mathbf e$ for the equilateral metric on four points,
\[
 \Gamma(d)=m\Gamma(\mathbf e),
 \qquad
 \Gamma(d^{\circ2})=m^2\Gamma(\mathbf e),
 \qquad
 \det\Gamma(\mathbf e)=\frac12.
\]
Hence equality in \eqref{eq:fourpoint} follows directly.
\end{proof}

\subsection{Main result}  
We are now ready to prove the $\FCC$ theorem.  
    
\begin{theorem}\label{thm:fivefacet}

For every nondegenerate point $(0,b,c,d,e,t)$,
\[
 F(0,b,c,d,e,t)\ge3\,2^{5/6}.
\]
Equality holds if and only if $(0,b,c,d,e,t)=(0,m,m,m,m,0)$ for some $m>0$.
\end{theorem}
\begin{proof}
Let $\mu_1,\mu_2,\mu_3\ge0$ be the eigenvalues of
$\mathsf G^{-1/2}\mathsf H\mathsf G^{-1/2}$.  Then
\[\tr \mathsf G^{-1/2}\mathsf H\mathsf G^{-1/2}=\mu_1+\mu_2+\mu_3\quad \text{and}\quad \det\mathsf H=(\det\mathsf G) \mu_1\mu_2\mu_3.\]
Therefore from \eqref{nuovaF} and from the inequality  
$\mu_1+\mu_2+\mu_3\geq 3 \sqrt[3]{\mu_1\mu_2\mu_3}$ 
we get
\[F(\rho)=2\det\mathsf G^{1/6}(\mu_1+\mu_2+\mu_3)\geq 6 \det\mathsf G^{1/6}  \sqrt[3]{\mu_1\mu_2\mu_3
}=6 \left(\frac{(\det\mathsf H)^2}{\det\mathsf G}\right)^{1/6}.\ \]
Theorem~\ref{thm:fourpoint} yields
$2(\det\mathsf H)^2\ge\det\mathsf G $, so, substituting in the previous inequality we get 
\[F(\rho)\geq 6 \left(\frac{(\det\mathsf H)^2}{\det\mathsf G}\right)^{1/6}\geq 3 \ 2^{5/6}=F_{\FCC}. \] This gives the minimality of $\FCC$. 
 
If equality holds in the final estimate, then equality holds both in the
spectral arithmetic--geometric mean and in Theorem~\ref{thm:fourpoint}.
Since $D(\rho)=\det\mathsf G>0$, the latter theorem forces all six metric lengths
$d_{ij}$, and hence all six $c_i$, $i\neq 4$ to agree. The relations
$c_3=c_5$, $c_3=c_6$, $c_0=c_1$, and $c_1=c_2$ imply 
$b=c=d=e$ and $t=0$.  Conversely, on the FCC ray all six $c_i$, for $i\neq 4$,
are equal and $\mathsf H$ is a positive scalar multiple of $\mathsf G$, so
there is equality in both constituent inequalities.
\end{proof}


\begin{theorem}[FCC minimality]\label{thm:FCC}
Let $\Vor(\Lambda)$ be the Voronoi cell of a full-rank lattice in $\R^3$.
If $\Vor(\Lambda)$ has at most twelve facets, then
\begin{equation}\label{eq:FCC-main}
\Iso(\Vor(\Lambda))=F(\rho)\ge3\,2^{5/6}.
\end{equation}
Equality holds if and only if $\Vor(\Lambda)$ is similar to the regular rhombic
dodecahedron, equivalently if and only if the lattice is similar to FCC.
\end{theorem}
\begin{proof}
Corollary~\ref{cor:twelve-facets-boundary} identifies the stated class with the
nondegenerate boundary of the Selling cone.  Proposition~\ref{prop:four-generators}
handles strata with at most four active parameters, and
Theorem~\ref{thm:fivefacet} handles the six open five-generator facets.
Their equality classifications agree exactly on the FCC orbit.
\end{proof}

\section{The BCC theorem}\label{sec:global-close}
The aim of this section is to prove the sharp determinant inequality  \eqref{eq:sharpdet} that, as we noted at the end of Subsection~\ref{esplicita}, 
controls the interior of the Selling cone and singles out the BCC ray.  

We do
not minimize the  determinant $\cJ(\rho)$ directly in the six Selling
coordinates.  Instead, the seven cut directions introduced above  in \eqref{zi} allow us to
separate the argument into three steps.  

First, by Cauchy--Binet we rewrite  the
determinant $\cJ$ as a sum over the $29$ nonzero bases of these directions;
grouping the bases according to the number of double cuts produces a
four-state lower bound.  Second, the geometry of the positive Selling cone is
encoded by two inequalities for the three double-cut leverage scores: the
product estimate in Lemma~\ref{thm:leverage-product} and the correlation
estimate in Proposition~\ref{thm:weak-correlation}.  Finally, these constraints
reduce the remaining optimization to the scalar estimate of
Theorem~\ref{thm:four-state-scalar}.  

In this way, we separate the algebraic
steps   and the  final optimization result.

\subsection{Determinantal probability measures}  \label{sec:compact-entropy}

First of all we show how to associate to the weighted spanning-tree polynomials $(q_i)_{i=0,\dots,6}$ be   defined
in \eqref{qdef} a determinantal probability measure. 
 
Let $z_i$   the seven cut directions   introduced  in \eqref{zi}. 
For every $u\in \mathbb R^7$, let 
\[ A_u=\sum_i u_i z_i z_i^\trans=\begin{pmatrix} u_0+u_1+u_5+u_6& u_0+u_6 &u_0+u_5\\
 u_0+u_6 & u_0+u_2+u_4+u_6 & u_0+u_4\\
 u_0+u_5 &u_0+u_4&u_0+u_3+u_4+u_5
 \end{pmatrix}\] 
 and 
 \begin{equation}
 \cP(u):=\det A_u.\label{eq:Pdef} 
\end{equation}
Note that if $u=(q_i)_{i=0,\dots,6}$  the weighted spanning-tree polynomials defined
in \eqref{qdef}, direct substitution gives the matrix identity
\begin{equation}\label{eq:physical-adjugate-lattice}
A_q=\sum_{i=0}^6 q_i z_i z_i^{\mathsf T}
=\operatorname{adj}A(\rho)
=(\det A(\rho))A(\rho)^{-1}.
\end{equation}

For every $I\subseteq \{0,
\dots, 6\}$, with $|I|=3$, $Z_I$  is  the $3\times3$ matrix whose columns are the vectors $z_i$, $i\in I$, and define 
\[\kappa_I=(\det Z_I)^2.\]
Of the $\binom73=35$ triples $I$, exactly $29$ matrices $Z_I$ have nonzero
determinant.  For $28$ of them the square of the  determinant  is $1$, while
$(\det(z_4,z_5,z_6))^2 =4$.
 
The Cauchy--Binet theorem gives
\begin{equation}\label{cbformula}
 \cP(u)=
\sum_{\substack{I\subset\{0,\ldots,6\}\\ |I|=3}}
 \det(Z_I)^2\prod_{i\in I}u_i=\sum_{\substack{I\subset\{0,\ldots,6\}\\ |I|=3}}
 \kappa_I\prod_{i\in I}u_i.
\end{equation}
    
From the Cauchy-Binet formula above one derives the  explicit  expression 
\begin{align}
\cP(u)={}&e_3(u_0,u_1,u_2,u_3)
+u_4(u_0+u_1)(u_2+u_3)\notag\\
&+u_5(u_0+u_2)(u_1+u_3)
+u_6(u_0+u_3)(u_1+u_2)\notag\\
&+(u_0+u_1+u_2+u_3)e_2(u_4,u_5,u_6)
+4u_4u_5u_6.\label{eq:Pexplicit}
\end{align}

To every fixed $u\in\R^7$ with $u_i\geq 0$ and $\cP(u)>0$ we associate a probability measure on the space $\Omega:=\{I\subset \{0,\ldots,6\}, |I|=3, \kappa_I>0\} $  defined as follows: 
\begin{equation}\label{eq:tau-u}
\tau_u(I)=\frac{\kappa_I\prod_{i\in I}u_i}{\cP(u)},
 \qquad |I|=3,
 \qquad \kappa_I:=\det(Z_I)^2.
\end{equation}
Note that $\tau_u $ is a probability measure due to \eqref{cbformula}.    \(\tau_u\) is the probability law on the
$29$ triples for which \(\kappa_I>0\) measuring    the relative contribution
of each basis to the determinant \(\cP(u)\) and it  is called the finite projection
determinantal measure associated with the rank-three subspace spanned by the weighted columns; see \cite[Sections~2--3]{Lyons2003}.  The one-point
marginals  associated to $\tau_u$ are defined for every $i\in\{0,
\dots, 6\}$ as 
\begin{equation}\label{eq:marginal-u} 
\ell_i(u):=\sum_{I\text{ s.t. } i\in I}\tau_u(I).
\end{equation} It is straightforward to check using the definition that  
\begin{equation}\label{espressionel_i}
\ell_i(u)= \frac{u_i}{\cP(u)}\frac{\partial \cP(u)}{\partial u_i }=u_i\frac{\partial}{\partial u_i}\log\cP(u).\end{equation}
By Jacobi's formula, recalling that
$\operatorname{adj}A_u=\cP(u)A_u^{-1}$ and
$\frac{\partial A_u}{\partial u_i}=z_i z_i^{\mathsf T}$, we obtain
\begin{align}
\label{formula1}
u_i\frac{\partial}{\partial u_i}\cP(u)
&=
u_i\tr\left(
\operatorname{adj}A_u\,
\frac{\partial A_u}{\partial u_i}
\right) \notag\\
&=
u_i\cP(u)\,z_i^{\mathsf T}A_u^{-1}z_i,
\\
\label{formula2}
\ell_i(u)
=u_i\frac{\partial}{\partial u_i}\log\cP(u)
&=u_i z_i^{\mathsf T}A_u^{-1}z_i.
\end{align}

For $u=q= (q_i)_{i=0,\dots,6}$ be the weighted spanning-tree polynomials defined
in \eqref{qdef} we  define $\cP(q)$ as in \eqref{eq:Pdef}, $\tau_q$ as in
\eqref{eq:tau-u}, and denote its one-point marginals by $\ell_i(q)$.
From the computation above, we deduce easily the following  identities: 
\begin{lemma}\label{lem:Pq}
There hold
\begin{equation}\label{eq:Pq}
\cP(q)=\det\operatorname{adj}A(\rho)=D(\rho)^2,
\qquad
\ell_i(q)=\frac{q_i c_i}{D(\rho)},
\qquad
q_i\frac{\partial}{\partial q_i}\cP(q)
=D(\rho)^2\ell_i(q).
\end{equation}
\end{lemma}

\begin{proof}
For a $3\times3$ matrix $A$,
\[
\det(\operatorname{adj}A)=(\det A)^2,
\qquad
\operatorname{adj}(\operatorname{adj}A)=(\det A)A.
\]
Since $\det A(\rho)=D(\rho)$, the first identity follows from
\eqref{eq:physical-adjugate-lattice}.  Moreover,
$A_q^{-1}=D(\rho)^{-1}A(\rho)$, and hence, using
$z_i^{\mathsf T}A(\rho)z_i=c_i$ and \eqref{formula2},
\[
\ell_i(q)
=q_i z_i^{\mathsf T}A_q^{-1}z_i
=\frac{q_i c_i}{D(\rho)}.
\]
Finally, \eqref{formula1} gives
\[
q_i\frac{\partial}{\partial q_i}\cP(q)
=\cP(q)\ell_i(q)
=D(\rho)^2\ell_i(q).
\]
\end{proof}


From now on we denote the  marginals in \eqref{eq:Pq} as the leverage scores 
\begin{equation}\label{eq:ell}
 \ell_i=\frac{q_ic_i}{D(\rho)}, \qquad \sum_{i=0}^6\ell_i=3.
\end{equation} 
 
We conclude this section with a simple observation that will be useful. Let 
\begin{equation}\label{eq:L-T}
 L=\ell_0+\ell_1+\ell_2+\ell_3,
 \qquad
 T=abc+ade+bdf+cef=D(\rho)
 \left[1-\ell_4-\ell_5-\ell_6\right].
\end{equation}
The spanning-tree identity \eqref{eq:tree} gives
\begin{equation}\label{eq:Lminus2}
 L-2=1-(\ell_4+\ell_5+\ell_6)= \frac{T}{D(\rho)}.
\end{equation}

We observe the following property. 
\begin{lemma} 
\label{thm:leverage-product}
For every positive Selling vector $\rho$,
\begin{equation}\label{eq:leverage-product}
\ell_4\ell_5\ell_6\le\frac{(1-\ell_4-\ell_5-\ell_6)^2}{4}=\frac{(L-2)^2}{4}.
\end{equation}
Equality holds precisely on the opposite-edge family
\begin{equation}\label{eq:opposite}
 a=f,\qquad b=e,\qquad c=d.
\end{equation}
\end{lemma}

\begin{proof}
Since \(q_4q_5q_6=abcdef\), due to \eqref{eq:Lminus2}, inequality
\eqref{eq:leverage-product} is equivalent to
\begin{equation}\label{eq:poly-leverage}
 4abcdef\,c_4c_5c_6\le T^2D(\rho).
\end{equation}
Let us introduce the three opposite-pair sums and differences
\[
 S_a=a+f,\quad S_b=b+e,\quad S_c=c+d,
 \qquad
 x=a-f,\quad y=b-e,\quad z=c-d,
\]
and normalize
\[
 \alpha=\frac{x}{S_a},\qquad \beta=\frac{y}{S_b},\qquad \gamma=\frac{z}{S_c}.
\]
Then \(|\alpha|,|\beta|,|\gamma|<1\).  Put
\[
 K=(S_a+S_b)(S_a+S_c)(S_b+S_c),
\]
\[
 \lambda_a=\frac{S_a^2(S_b+S_c)}K,
 \quad
 \lambda_b=\frac{S_b^2(S_a+S_c)}K,
 \quad
 \lambda_c=\frac{S_c^2(S_a+S_b)}K,
 \quad
 \eta=\frac{S_aS_bS_c}{K}.
\]
Note that \begin{equation}\label{eq:weight-sum}
 \lambda_a+\lambda_b+\lambda_c+2\eta=1.
\end{equation}
Then by direct computations we obtain 
\begin{equation*}
 T=\frac{S_aS_bS_c+xyz}{2}=\frac{S_aS_bS_c}{2}(1+\alpha\beta\gamma),
\end{equation*}
\begin{equation*}
 D(\rho)=\frac K4\Theta,
\end{equation*}
where
\begin{equation}\label{eq:pair-Theta}
 \Theta=\lambda_a(1-\alpha^2)+\lambda_b(1-\beta^2)
 +\lambda_c(1-\gamma^2)+2\eta(1+\alpha\beta\gamma).
\end{equation}

Furthermore
\[
abcdef=\frac{S_a^2S_b^2S_c^2}{64}
 (1-\alpha^2)(1-\beta^2)(1-\gamma^2),
 \qquad
 c_4c_5c_6=K.
\]
Therefore proving \eqref{eq:poly-leverage} (and so \eqref{eq:leverage-product}) is equivalent to prove 
\begin{equation}\label{eq:dimensionless}
 (1+\alpha\beta\gamma)^2\Theta
 \ge(1-\alpha^2)(1-\beta^2)(1-\gamma^2).
\end{equation}

First of all by weighted AM--GM and \eqref{eq:weight-sum}, recalling  \eqref{eq:pair-Theta}
\begin{equation}\label{theta}
 (1-\alpha^2)^{\lambda_a}(1-\beta^2)^{\lambda_b}
 (1-\gamma^2)^{\lambda_c}(1+\alpha\beta\gamma)^{2\eta}
\leq \lambda_a(1-\alpha^2)+\lambda_b(1-\beta^2)+\lambda_c
 (1-\gamma^2) +2\eta(1+\alpha\beta\gamma)=\Theta. \end{equation}
Now we observe that  for every  \(|h|,|k|,|g|<1\),
\begin{equation}\label{eq:pair-AMGM}
 (1-k^2)(1-g^2)\le(1+hkg)^2.
\end{equation}
 Indeed
\[
 1+hkg\ge1-|kg|,
 \qquad
 (1-|kg|)^2-(1-k^2)(1-g^2)=(|k|-|g|)^2\geq 0.
\]
 Therefore, using \eqref{eq:pair-AMGM} for $\alpha,\beta,\gamma$ we get
 \begin{eqnarray*} (1-\beta^2)^{\lambda_a+\eta}(1-\gamma^2)^{\lambda_a+\eta}
\leq (1+\alpha\beta \gamma)^{2\lambda_a+2\eta}
\\  (1-\gamma^2)^{\lambda_b+\eta}(1-\alpha^2)^{\lambda_b+\eta}
\leq (1+\alpha\beta \gamma)^{2\lambda_b+2\eta}\\
 (1-\beta^2)^{\lambda_c+\eta}(1-\alpha^2)^{\lambda_c+\eta}
\leq (1+\alpha\beta \gamma)^{2\lambda_c+2\eta}
\end{eqnarray*}
Multiplying the three lines and recalling \eqref{eq:weight-sum}, we get 
\[(1-\alpha^2)^{1-\lambda_a}(1-\beta^2)^{1-\lambda_b}(1-\gamma^2)^{1-\lambda_c}\leq (1+\alpha\beta \gamma)^{2+2\eta}
\]
   Multiplying this inequality by \eqref{theta} 
\[
 (1+\alpha\beta\gamma)^2\Theta\ge
 (1-\alpha^2)(1-\beta^2)(1-\gamma^2),
\]
which is \eqref{eq:dimensionless}. This gives the conclusion. 

All four AM--GM weights $\lambda_a, \lambda_b, \lambda_c, \eta$ are positive in the interior.  Equality in \eqref{theta} occurs when
\[
1-\alpha^2=1-\beta^2=1-\gamma^2=1+\alpha\beta\gamma,
\]
which implies $\alpha=\beta=\gamma=0$.  This gives $a=f$, $b=e$, and
$c=d$, that is \eqref{eq:opposite}.
\end{proof}


\subsection{A weak correlation estimate}\label{sec:weak-correlation}

We provide now a parametrization  of  $q_i$ in \eqref{qdef} by six positive
variables and will be used in the correlation estimate below.

\begin{lemma} \label{thm:q-chart}
Set
\begin{equation}\label{eq:chart-products}
 A=af,\quad B=be,\quad C=cd,
 \qquad X=ef,\quad Y=df,\quad Z=de
\end{equation}
so that 
\begin{align}
 q_0&=X+Y+Z,\label{eq:q0chart}\\
 q_1&=\frac{BC+CX+BY}{Z},\label{eq:q1chart}\\
 q_2&=\frac{AC+CX+AZ}{Y},\label{eq:q2chart}\\
 q_3&=\frac{AB+BY+AZ}{X},\label{eq:q3chart}\\
 q_4&=A,\qquad q_5=B,\qquad q_6=C.\label{eq:qdoublechart}
\end{align}
Conversely, given  \((A,B,C,X,Y,Z)\in\R^6\) with $A,B,C,X,Y,Z> 0$ it is possible to define a positive 
Selling vector $(a,b,c,d,e,f)$ by
\begin{equation}\label{eq:chart-inverse}
 d=\sqrt{\frac{YZ}{X}},\quad
 e=\sqrt{\frac{XZ}{Y}},\quad
 f=\sqrt{\frac{XY}{Z}},\qquad
 a=\frac Af,\quad b=\frac Be,\quad c=\frac Cd.
\end{equation}
Therefore \eqref{eq:q0chart}--\eqref{eq:qdoublechart} parameterize exactly
the image of the positive Selling cone in the space of weighted spanning-tree polynomials $q_i$.
\end{lemma}

\begin{proof}
It is a straightforward computation, by using definition \eqref{qdef}. 
\end{proof}
We rewrite also the leverage scores in terms of the new variables as follows. 

Define
\begin{align}
\Delta={}&ABC+ABY+ABZ+ACX+ACZ+AXZ+AYZ+AZ^2\notag\\
&+BCX+BCY+BXY+BY^2+BYZ+CX^2+CXY+CXZ,
\label{eq:chart-Delta}
\end{align}
and
\begin{align*}
h_4&=A(BY+CX+XZ+YZ),\\
h_5&=B(AZ+CX+XY+YZ),\\
h_6&=C(AZ+BY+XY+XZ).
\end{align*}
 
Substituting
\eqref{eq:q0chart}--\eqref{eq:qdoublechart}
into \eqref{eq:Pexplicit} and collecting terms gives
\[
 \cP(q)=\frac{\Delta^2}{XYZ}.
\]
Since \(\cP(q)=D(\rho)^2\) by \eqref{eq:Pq}, and all quantities are
positive on the positive Selling cone, it follows 
\[
 D(\rho)=\frac{\Delta}{\sqrt{XYZ}}.
\]
Moreover, using \eqref{eq:ell}, \eqref{cutw},
\eqref{eq:qdoublechart}, and \eqref{eq:chart-inverse}, we obtain 
\[
\begin{aligned}
 \ell_4
  =\frac{q_4c_4}{D(\rho)}
 =\frac{A(b+c+d+e)\sqrt{XYZ}}{\Delta} =\frac{A(BY+CX+XZ+YZ)}{\Delta}
 =\frac{h_4}{\Delta}.
\end{aligned}
\]
The identities
\[
 \ell_5=\frac{h_5}{\Delta},
 \qquad
 \ell_6=\frac{h_6}{\Delta}
\]
follow in the same way.
Therefore we have  
\begin{equation}\label{eq:chart-covariants}
 \cP(q)=\frac{\Delta^2}{XYZ},
 \qquad
 \ell_i=\frac{h_i}{\Delta}\quad (i=4,5,6).
\end{equation}
Let now
\begin{align}\label{eq:H-C-L}
 H&=h_4+h_5+h_6=\Delta(\ell_4+\ell_5+\ell_6),
\\
 \mathscr L&=\Delta(\Delta-H)^2-4h_4h_5h_6= \Delta^3\left[ (1-\ell_4-\ell_5-\ell_6)^2-4 \ell_4\ell_5\ell_6\right]\notag \\&=\Delta^3[(L-2)^2-4 \ell_4\ell_5\ell_6]\geq 0,
\notag\end{align}
where $L-2$ is defined in \eqref{eq:Lminus2} and the last inequality is proved in  \eqref{eq:leverage-product}. 
We define also 
\begin{align}
\mathscr C:={}&A^2BCZ^2+AB^2CY^2+ABC^2X^2
 +ABY^2Z^2\notag\\
&+ACX^2Z^2+BCX^2Y^2-6ABCXYZ.
\label{eq:correlation-defect}
\end{align}
We show that $\mathscr C\geq 0$.
\begin{lemma}\label{lem:projection-kernel}
Let $\tau_q$ the probability measure   associated to $(q_i)_i$.
Let \[
 J_2=\sum_{1\le i<j\le3}\sum_{ \stackrel{ I \text{ s.t. }  }{i+3,j+3\in I}}\tau_q(I) = 
 \sum_{  I \text{ s.t. }   {4,5\in I}}\tau_q(I) +\sum_{  I \text{ s.t. }   {4,6\in I}}\tau_q(I) +\sum_{  I \text{ s.t. }   {5,6\in I}}\tau_q(I)
\] and \[
 J_3= \tau_q(4,5,6).
\]
 
Then
\begin{equation}\label{eq:J2J3}
 J_2=\ell_4\ell_5+\ell_4,\ell_6+\ell_5\ell_6-\frac{\mathscr C}{\Delta^2},
 \qquad
 J_3\le \ell_4\ell_5\ell_6.
\end{equation} 
In particular $\mathscr C\ge0$.
\end{lemma}

\begin{proof}

By \eqref{eq:tau-u} and recalling the definition of  \(\cP\), for \(i\ne j\)
we observe that
\[
 \sum_{   I \text{ s.t. }  i,j\in I}\tau_q(I)
 =
 \frac{q_iq_j\,\partial_{ij}\cP(q)}{\cP(q)},
 \qquad
 \partial_{ij}:=\frac{\partial^2}{\partial q_i\partial q_j}.
\]
Hence
\[
 J_2
 =
 \frac{
 q_4q_5\,\partial_{45}\cP(q)
 +q_4q_6\,\partial_{46}\cP(q)
 +q_5q_6\,\partial_{56}\cP(q)
 }{\cP(q)}.
\]
Using the explicit formula \eqref{eq:Pexplicit}, 
substituting
\eqref{eq:q0chart}--\eqref{eq:qdoublechart},
and collecting terms gives
\[
 q_4q_5\,\partial_{45}\cP(q)
 +q_4q_6\,\partial_{46}\cP(q)
 +q_5q_6\,\partial_{56}\cP(q) 
  =
 \frac{
 h_4h_5+h_4h_6+h_5h_6-\mathscr C
 }{XYZ}.
\]
Using \eqref{eq:chart-covariants}, we therefore obtain
\[
 J_2
 =
 \frac{h_4h_5+h_4h_6+h_5h_6}{\Delta^2}
 -\frac{\mathscr C}{\Delta^2}=
\ell_4\ell_5+\ell_4,\ell_6+\ell_5\ell_6-\frac{\mathscr C}{\Delta^2},
\]
which proves the first identity in \eqref{eq:J2J3}.

Moreover,  recalling  that
$\ell_j=\sum_{ I \text{ s.t.    }j \in I}\tau_q(I) $,  
we get \[
 \frac{\mathscr C}{\Delta^2}
 =
 \sum_{1\le i<j\le3}
 \left(
 \ell_{i+3}\ell_{j+3}-\sum_{ \stackrel{ I \text{ s.t. }  }{i+3,j+3\in I}}\tau_q(I) 
 \right)\ge0.
\]
  Thus \(\mathscr C\ge0\).

Finally, 
Hadamard's inequality gives
\[
 J_3\le \ell_4\ell_5\ell_6,
\]
completing the proof.
\end{proof}
We have the following weak correlation estimate between one-point marginals $\ell_i$, $i=4,5,6$ and the two-point marginal $J_2$. 
\begin{proposition}\label{thm:weak-correlation}
There holds 
\begin{equation}\label{eq:correlation-control}
\ell_4\ell_5+\ell_4,\ell_6+\ell_5\ell_6-J_2\le\frac{(1-(\ell_4+\ell_5+\ell_6))^2-4\ell_4\ell_5\ell_6}{1-(\ell_4+\ell_5+\ell_6)}.\end{equation}

Equality  in \eqref{eq:correlation-control} is possible only on the opposite-edge
locus
\[a=f\qquad b=e\qquad c=d.\]
\end{proposition}

\begin{proof}
In order to prove \eqref{eq:correlation-control}, we rewrite it as follows: \begin{equation}\label{eq:stability-weak}
 \mathscr L\ge(\Delta-H)\mathscr C
\end{equation} recalling the definition of  $\mathscr L, \Delta, H, \mathscr C$  in \eqref{eq:H-C-L}, \eqref{eq:chart-Delta}, \eqref{eq:correlation-defect} and the fact that by  \eqref{eq:J2J3} we get 
$\mathscr C=\Delta^2(\ell_4\ell_5+\ell_4,\ell_6+\ell_5\ell_6-J_2.$

Write
\[
 A=\mathfrak a^2,\quad B=\mathfrak b^2,\quad C=\mathfrak c^2,
 \qquad
 X=\mathfrak a\mathfrak b x,\quad
 Y=\mathfrak a\mathfrak c y,\quad
 Z=\mathfrak b\mathfrak c z,
\]
with $x,y,z>0$.  A direct collection of
$\mathscr L-(\Delta-H)\mathscr C$ in the scale variables gives
\begin{equation}\label{eq:weak-collection}
 \mathscr L-(\Delta-H)\mathscr C
 =\mathfrak a^5\mathfrak b^5\mathfrak c^5\,\Psi,
\end{equation}
where
\begin{align*}
\Psi={}&
 \mathfrak a^2\mathfrak b\,(xz+y)E_z^-E_z^+
 +\mathfrak a^2\mathfrak c\,(x+yz)E_z^-E_z^+\\
&+\mathfrak a\mathfrak b^2\,(xy+z)E_y^-E_y^+
 +\mathfrak b^2\mathfrak c\,(x+yz)E_y^-E_y^+\\
&+\mathfrak a\mathfrak c^2\,(xy+z)E_x^-E_x^+
 +\mathfrak b\mathfrak c^2\,(xz+y)E_x^-E_x^+\\
&+\mathfrak a\mathfrak b\mathfrak c\,Q(x,y,z),
\end{align*}
with
\begin{align*}
 E_z^-&=(x-y)^2+(z-1)^2,&
 E_z^+&=(x+y)^2+(z+1)^2,\\
 E_y^-&=(x-z)^2+(y-1)^2,&
 E_y^+&=(x+z)^2+(y+1)^2,\\
 E_x^-&=(y-z)^2+(x-1)^2,&
 E_x^+&=(y+z)^2+(x+1)^2.
\end{align*}
It remains only to show $Q\ge0$.  Set
\[
 s_1=x^2+y^2+z^2,\qquad
 s_2=x^2y^2+y^2z^2+z^2x^2,\qquad
 r=xyz.
\]
The remaining symmetric coefficient is
\begin{equation}\label{eq:weak-Q}
 Q=s_1^3+2s_1^2+2s_1+1-(s_1+9)s_2-8r^2-2r(s_1+1).
\end{equation}
For the three positive numbers $x^2,y^2,z^2$,
\[
 s_2\le\frac{s_1^2}{3},\qquad
 r^2\le\left(\frac{s_1}{3}\right)^3,
 \qquad
 r\le\left(\frac{s_1}{3}\right)^{3/2}.
\]
Since all three quantities occur in \eqref{eq:weak-Q} with negative
coefficients,
\[
 Q\ge
 \frac{10}{27}s_1^3-s_1^2+2s_1+1
 -2(s_1+1)\left(\frac{s_1}{3}\right)^{3/2}.
\]
Putting $\tau=\sqrt{s_1/3}$, the right-hand side factors as
\[
 (\tau-1)^2
 \left(10\tau^4+14\tau^3+9\tau^2+2\tau+1\right)\ge0.
\]
Thus every term in \eqref{eq:weak-collection} is nonnegative, proving
\eqref{eq:stability-weak}.  
If equality holds, the six manifestly nonnegative scale terms force
$x=y=z=1$, that is  $X^2=AB, Y^2=AC,  Z^2=BC$, which in turns give
$ef=ab$, $df=ac$ and $de=bc$, i.e. $a=f, b=e, c=d$. which is the stated opposite-edge locus.
\end{proof}


Let $\mathcal B_k=\{I: I\subset\{0, \dots, 6\}, |I|=3,\kappa_I>0,\ |I\cap\{4,5,6\}|=k\}$ and define 
\[
P_k=\sum_{I\in \mathcal B_k }\tau_q(I),\quad 0\le k\le3.
\]
Recalling the definition of $J_2,J_3$  in Lemma~\ref{lem:projection-kernel}, and the fact the $\ell_i$
are one-point marginals, we may explicit determine $P_k$ as follows: \begin{equation}\label{eq:Pk-J}
 \begin{aligned}
 P_3&= J_3,&\qquad
 P_2&=J_2-3J_3\\ 
 P_1&=\ell_4+\ell_5+\ell_6-2J_2+3J_3, & \qquad 
 P_0&=1-\sum_{i=1}^3P_i=1- (\ell_4+\ell_5+\ell_6)+J_2-J_3.
 \end{aligned}
\end{equation}

We provide an inequality relating the determinant $\cJ(\rho)$ and a suitable function of $P_i$. 
\begin{lemma}\label{lem:four-state}
For every positive Selling vector $\rho$,
\begin{equation}\label{eq:four-state}
 \cJ(\rho)\ge  
 \frac{4P_0^{3/2}}{(L)^{3/2}}
 +\frac{2P_1^{3/2}}{L\sqrt {\ell_4+\ell_5+\ell_6}}
 +\frac{P_2^{3/2}}{\sqrt{L(\ell_4\ell_5+\ell_5\ell_6+\ell_4\ell_6)}}
 +\frac{P_3^{3/2}}{2\sqrt{\ell_4\ell_5\ell_6}}:=\Phi(\rho)
\end{equation}
where $L=\ell_0+\ell_1+\ell_2+\ell_3$ (see \eqref{eq:L-T}). 
At BCC (that is $a=b=c=d=e=f=1$) equality holds in \eqref{eq:four-state}.
\end{lemma}

\begin{proof}
For each  of the $29$ bases \(I\) with \(\kappa_I>0\), put
\[
 a_I=\kappa_I\prod_{i\in I}\ell_i.
\]
By Cauchy--Binet formula  and \eqref{eq:M-J},
\[
 \cJ(\rho)
 =
 \frac{1}{D^{3/2}(\rho)}
 \sum_{I:\,\kappa_I>0}
 \kappa_I\prod_{i\in I}\frac{q_i}{\sqrt{c_i}} =
 \sum_{I:\,\kappa_I>0}
 \frac{
 \kappa_I\prod_{i\in I}q_i
 }{
 D^{3/2}(\rho)\sqrt{\prod_{i\in I}c_i}
 }.
\]
On the other hand, by \eqref{eq:tau-u}, \eqref{eq:Pq}, and
\eqref{eq:ell},
\[
 \tau_q(I)
 =
 \frac{\kappa_I\prod_{i\in I}q_i}{D^2(\rho)},
 \qquad
 a_I
 =
 \frac{
 \kappa_I\prod_{i\in I}q_ic_i
 }{D^3(\rho)}.
\]
Therefore
\[
 \frac{\tau_q(I)^{3/2}}{\sqrt{a_I}}
 =
 \frac{
 \kappa_I\prod_{i\in I}q_i
 }{
 D^{3/2}(\rho)\sqrt{\prod_{i\in I}c_i}
 }.
\]
As a consequence,
\begin{equation}\label{eq:J-Renyi-sum}
 \cJ(\rho)
 =
 \sum_{I:\,\kappa_I>0}
 \frac{\tau_q(I)^{3/2}}{\sqrt{a_I}}.
\end{equation}

  Recalling the definition of $\mathcal B_k$ and $P_k$ above,    H\"older's inequality gives \[P_k^{3/2}=\left(\sum_{I\in\mathcal B_k}\tau_q(I)\right)^{\frac32} \leq 
\left(\sum_{I\in\mathcal B_k}a_I\right)^{\frac12}\sum_{I\in\mathcal B_k}\frac{\tau_q(I)^{3/2}}{\sqrt{a_I}}\] therefore letting $M_k:=\sum_{I\in\mathcal B_k}a_I$ we obtain 
\begin{equation}\label{eq:group-holder}\cJ(\rho) \geq \sum_{k=0}^{3}\frac{P_k^{3/2}}{\sqrt{M_k}}.
\end{equation}
From Cauchy-Binet formula as in ~\eqref{eq:Pexplicit}
we compute explicitly
\[ \begin{aligned}
 M_0&=\ell_0\ell_1\ell_2+\ell_0\ell_1\ell_3+\ell_0\ell_2\ell_3+\ell_1\ell_2\ell_3,\\
 M_1&=\ell_4(\ell_0+\ell_1)(\ell_2+\ell_3)
       +\ell_5(\ell_0+\ell_2)(\ell_1+\ell_3)+\ell_6(\ell_0+\ell_3)(\ell_1+\ell_2),\\
     M_2&=L(\ell_4\ell_5+\ell_4\ell_6+\ell_5\ell_6),\\M_3&=4\ell_4\ell_5\ell_6.  
 \end{aligned}
\]
 
From 
Maclaurin's inequality 
\[
M_0=\ell_0\ell_1\ell_2+\ell_0\ell_1\ell_3+\ell_0\ell_2\ell_3+\ell_1\ell_2\ell_3
\le
\binom{4}{3}
\left(
\frac{\ell_0+\ell_1+\ell_2+\ell_3}{4}
\right)^3
=
\frac{L^3}{16}
\]
and from the inequality
$(\ell_i+\ell_j)(L-\ell_i-\ell_j)\le L^2/4$, for $i\neq j\in\{0,1,2,3\}$ we get
\begin{equation}\label{eq:Mk-upper}
 M_1\le\frac{L^2(\ell_4+\ell_5+\ell_6)}{4}.
\end{equation}
If we substitute in    \eqref{eq:group-holder}, we obtain 
proves \eqref{eq:four-state}.  

Observe now that at $\BCC$, there hold: $\ell_i=\frac{9}{16}$ for $i=0,1,2,3$, $\ell_i=\frac{1}{4}$ for $i=4,5,6$ and  
\[
 (P_0,P_1,P_2,P_3)=\frac1{64}(27,27,9,1),
\]
and direct substitution gives $\Phi(1,1,1,1,1,1)=\cJ_{\BCC}=\cJ(1,1,1,1,1,1)$.
\end{proof}
Now we reduced to prove that $\BCC$ that is $\rho=(1,1,1,1,1,1)$
is a minimizer of the function $\Phi(\rho)$, defined in \eqref{eq:four-state}.

In the following proof we use repeatedly the elementary support inequality: for every $x,h\geq 0$, $m>0$ there holds
\begin{equation}\label{eq:support-32}
 \frac{x^{3/2}}{\sqrt m}
 \ge hx-\frac{4m}{27}h^3.
\end{equation}
Indeed the minimum of the function $\psi(x)=\frac{x^{3/2}}{\sqrt m}-hx$   occurs at
$x=4mh^2/9$.

\begin{theorem}\label{thm:four-state-scalar}
For every positive Selling vector $\rho$,
\begin{equation}\label{eq:four-state-C}
\Phi(\rho)\ge \Phi(1,1,1,1,1,1)=\cJ_{\BCC}=\frac{37+30\sqrt3}{128}
\end{equation}
where $\Phi$
 is defined in \eqref{eq:four-state}. Equality is possible only when $\rho=(1,1,1,1,1,1)$, that is
\begin{equation}\label{eq:four-state-equality} \ell_4=\ell_5=\ell_6=\frac14.
\end{equation}
\end{theorem}

\begin{proof} Recall by   Lemma~\ref{thm:leverage-product},
Propositio~\ref{thm:weak-correlation}, and
Lemma~\ref{lem:projection-kernel}.
The physical constraints used below are
\begin{equation}\label{eq:four-state-physical}
\ell_4\ell_5\ell_6\le\frac{(L-2)^2}{4},\qquad
 J_2\ge \ell_4\ell_5+\ell_4,\ell_6+\ell_5\ell_6-\frac{(L-2)^2-4\ell_4\ell_5\ell_6}{L-2},\qquad
J_3\le \ell_4\ell_5\ell_6
\end{equation} where $L=\ell_0+\ell_1+\ell_2+\ell_3$. 
  We also use Schur's inequality and the AM--GM bound 
\begin{equation}\label{eq:schur-nup}
 \ell_4\ell_5+\ell_4,\ell_6+\ell_5\ell_6 \le\frac{(\ell_4+\ell_5+\ell_6)^3+9\ell_4\ell_5\ell_6}{4(\ell_4+\ell_5+\ell_6)},\qquad \ell_4\ell_5+\ell_6\leq \frac{(\ell_4+\ell_5+\ell_6)^3}{27}.
\end{equation}
We divide the proof in $3$ cases, according to the value of $L-2$. 

\smallskip
\noindent\emph{Range I: $0<L-2\le1/8$.}
Apply \eqref{eq:support-32} to the first three terms of
\eqref{eq:four-state} with the common slope
\[
 h=\frac{3}{2\sqrt2},
\]
and discard the last nonnegative term.  Since
$4h^3/27=1/(4\sqrt2)$,
\[
 \Phi(\rho)\ge
 h(1-P_3)-\frac1{4\sqrt2}
 \left(\frac{L^3}{16}+\frac{L^2(\ell_4+\ell_5+\ell_6)}{4}+L(\ell_4\ell_5+\ell_4\ell_6+\ell_5\ell_6)\right).
\]
Using $P_3=J_3\le \ell_4\ell_5\ell_6\le (L-2)^2/4$ and, from
\eqref{eq:schur-nup},
we obtain
\begin{equation}\label{eq:boundary-support}
 \Phi(\rho)\ge B(L-2) \end{equation} where \[B(t):=
 -\frac{\sqrt2\,(t-4)(t+2)(t^2+10t-8)}{128(t-1)}.
\]
Moreover
\[
 B'(t)=
 -\frac{3\sqrt2\,t(t^3+4t^2-20t+24)}{128(t-1)^2}<0
 \qquad(0<t\le1/8),
\]
since $t^3+4t^2-20t+24>24-20/8>0$.  Hence $B(t)\ge B(1/8)$, where
\[
 B(1/8)=\frac{227137\sqrt2}{458752}.
\]
The elementary rational estimates
\[
 \sqrt2>\frac{140}{99},\qquad \sqrt3<\frac{26}{15}
\]
follow by squaring, and give
\[
 B(1/8)-\cJ_{\BCC}
 >\frac{7877}{1622016}>0.
\]
Thus the inequality is strict in Range I.

\smallskip
\noindent\emph{Ranges II--III: $1/8\le L-2<1$.}
Apply \eqref{eq:support-32} to all four terms in
\eqref{eq:four-state} with 
respectively $(x_0, x_1,x_2,x_3)=(P_0,P_1, P_2, P_3)$, 
\[(m_0,m_1,m_2,m_3)
 =\left(\frac{L^3}{16}, \frac{L^2(\ell_4+\ell_5+\ell_6)}{4}, L(\ell_4\ell_5+\ell_5\ell_6+\ell_4\ell_6), 4\ell_4\ell_5\ell_6\right)\] 
\begin{equation}\label{eq:BCC-support-slopes}
 (h_0,h_1,h_2,h_3)
 =\left(\frac2{\sqrt3},1,\frac{\sqrt3}{2},\frac34\right).
\end{equation}
These are precisely the four support slopes at BCC.  Thus
\begin{equation}\label{eq:BCC-support}
 \Phi(\rho)\ge
 \sum_{k=0}^3h_kP_k
 -\frac4{27}\left(
 \frac{L^3}{16}h_0^3+\frac{L^2(\ell_4+\ell_5+\ell_6)}{4}h_1^3
+L(\ell_4\ell_5+\ell_5\ell_6+\ell_4\ell_6)h_2^3+4\ell_4\ell_5\ell_6h_3^3\right).
\end{equation}
Substituting to $P_i$ the explicit formula in \eqref{eq:Pk-J} we obtain that the  coefficients of $J_2$ and $J_3$ in the right-hand side are respectively
\[
 \frac{-12+7\sqrt3}{6}>0,
 \qquad
 \frac{45-26\sqrt3}{12}<0.
\]
  Hence, recalling \eqref{eq:four-state-physical}, we    substitute $J_2$ with $\ell_4\ell_5+\ell_4,\ell_6+\ell_5\ell_6-\frac{(L-2)^2-4\ell_4\ell_5\ell_6}{L-2}$ and $J_3$ with $\ell_4\ell_5\ell_6$. 

After these substitutions the coefficient of $\ell_4\ell_5+\ell_4,\ell_6+\ell_5\ell_6 $ is
\[
 -\frac{\sqrt3\,(L-2)-19\sqrt3+36}{18}<0,
\]
so Schur's upper bound \eqref{eq:schur-nup} may also be used; here $36>19\sqrt3$.

Denote by $G(L-2,\ell_4\ell_5\ell_6)=G(t,p)$ the resulting right-hand side of
\eqref{eq:BCC-support} minus $\cJ_{\BCC}$.
A direct collection gives
\begin{equation}\label{eq:Gp-derivative}
 \frac{\partial G}{\partial p}
 =-\frac{N(t)}{24t(1-t)},
\end{equation}
where
\[
N(t)=(84-49\sqrt3)t^2+(-168+107\sqrt3)t+192-112\sqrt3.
\]
Since $84-49\sqrt3<0$, $N(t)$ is concave on $[0,1]$. Observing that 
$N'(1)=9\sqrt3>0$, it follows that $N(t)$ is also increasing on $[0,1]$. 
Moreover
\[
 N(1/8)=\frac{11028-6361\sqrt3}{64}>0,
\]
because $11028^2-3\cdot6361^2=229821>0$.
As a consequence we get
\begin{equation}\label{eq:Gp-negative}
 \frac{\partial G}{\partial p}<0
 \qquad(1/8\le t<1).
\end{equation}

If $1/8\le t\le1/4$, then $p\le t^2/4$, and
\eqref{eq:Gp-negative} gives
$G(t,p)\ge G(t,t^2/4)$.  The endpoint factors as
\begin{equation}\label{eq:G-left-factor}
 G\!\left(t,\frac{t^2}{4}\right)
 =\frac{(72-43\sqrt3)(4t-1)Q_-(t)}{3763584(t-1)},
\end{equation}
where
\[
 Q_-(t)=484t^3-12(345\sqrt3+119)t^2
 -9(2467\sqrt3+3326)t+1718-873\sqrt3.
\]
Since $72-43\sqrt3<0$ (because $72^2<3\cdot43^2$), it remains
only to note that $Q_-(t)<0$ on this interval.  Indeed
\[
 Q_-'(t)=3\{484t^2-(2760\sqrt3+952)t-(7401\sqrt3+9978)\}<0,
\]
because $t\le1/4$ gives $484t^2\le121/4<7401\sqrt3+9978$, and
\[
 Q_-(1/8)=-\frac{59409\sqrt3}{16}-\frac{261775}{128}<0.
\]
All signs in \eqref{eq:G-left-factor} therefore give
$G(t,t^2/4)\ge0$, with equality only at $t=L-2=1/4$.

Finally let $1/4\le L-2<1$.  Here $p\le s^3/27$, so
$G(t,p)\ge G(t,s^3/27)$.  This endpoint factors as
\begin{equation}\label{eq:G-right-factor}
 G\!\left(t,\frac{s^3}{27}\right)
 =\left(-\frac5{54}+\frac{13\sqrt3}{243}\right)
 \frac{t-1/4}{t}\,Q_+(t),
\end{equation}
where
\[
 Q_+(t)=t^3+\left(\frac{37}{4}+6\sqrt3\right)t^2
 +\left(\frac{9589}{16}+\frac{705\sqrt3}{2}\right)t
 -96-48\sqrt3.
\]
The prefactor is positive because $(26\sqrt3)^2>45^2$.
Furthermore $Q_+'(t)>0$ for $t>0$ and
\[
 Q_+(1/4)=\frac{3483}{64}+\frac{81\sqrt3}{2}>0.
\]
Thus \eqref{eq:G-right-factor} is nonnegative, again with equality only
at $t=1/4$.

We have proved \eqref{eq:four-state-C}.  Equality forces $L-2=1/4$ and,
from the decreasing $p$-step above,
$\ell_4\ell_5\ell_6=(L-2)^2/4=1/64$.  Since also $\ell_4+\ell_5+\ell_6=3/4$, equality in AM--GM gives
$\ell_4=\ell_5=\ell_6=1/4$, proving \eqref{eq:four-state-equality}.
\end{proof}

\subsection{Main result}
By Lemma \ref{lem:four-state} and Theorem \ref{thm:four-state-scalar} above, we deduce the determinantal inequality \eqref{eq:sharpdet}.
\begin{theorem}\label{thm:J-global}
For every positive Selling vector $\rho$,
\[
\cJ(\rho)\ge \cJ_{\BCC}=\frac{37+30\sqrt3}{128}.
\]
Equality holds only on the BCC ray.
\end{theorem}

\begin{proof}
Lemma~\ref{lem:four-state} and
Theorem~\ref{thm:four-state-scalar} give
\[
 \cJ(\rho)\ge\Phi(\rho)\ge \cJ_{\BCC}.
\]
If equality holds, Theorem~\ref{thm:four-state-scalar} gives
$\ell_4=\ell_5=\ell_6=1/4$.  In particular
\[
 \ell_4\ell_5\ell_6=\frac1{64}=\frac{(L-2)^2}{4},
\]
so equality holds in Lemma~\ref{thm:leverage-product}.  Its equality
classification gives the opposite-edge family
$a=f=A$, $b=e=B$, $c=d=C$.  Equality of the three double leverage scores
then reads
\[
 A^2(B+C)=B^2(A+C)=C^2(A+B).
\]
Since
\[
 A^2(B+C)-B^2(A+C)=(A-B)\{AB+C(A+B)\},
\]
and the bracket is positive, $A=B$; cyclically $A=B=C$.  Thus the Selling
vector lies on the BCC ray.  Conversely direct substitution gives equality
at BCC.
\end{proof}

\begin{theorem}[Global minimality of BCC]\label{thm:BCC-global}
For every nonnegative Selling vector $\rho$ defining a nondegenerate full-rank
lattice Voronoi cell,
\[
 F(\rho)\ge F_{\BCC}.
\]
Equality holds only on the BCC ray.  Therefore the regular truncated octahedron is the unique global minimizer
among all three-dimensional lattice Voronoi cells, up to similarity.
\end{theorem}

\begin{proof}
In the positive interior, \eqref{eq:spectral} and the preceding theorem give
\[
 F(\rho)\ge6\cJ^{1/3}(\rho)\ge6\cJ_{\BCC}^{1/3}=F_{\BCC}.
\]
Equality is attained only at BCC.  On the nondegenerate boundary,
Theorem~\ref{thm:FCC} gives $F(\rho)\ge3\,2^{5/6}>F_{\BCC}$.  The strict
comparison is exact: after cubing and cancelling the common positive factor it
is equivalent to
\[
 64\sqrt2>37+30\sqrt3.
\]
Indeed $\sqrt2>7/5$ and $\sqrt3<7/4$, so that
$64\sqrt2>448/5>179/2>37+30\sqrt3$.
\end{proof}

\section{Extension to all parallelohedra}
\label{sec:affine-extension}

This section develops the part of the Truncated Octahedron Conjecture which is
not contained in the lattice--Voronoi theorem.  The key point is that the
weighted $K_4$ combinatorics used throughout the paper survives unchanged:
what is lost outside the Voronoi subclass is only the relation tying the
Euclidean metric to the six graphical weights.

\subsection{A graphical parametrization}

We retain from Section~\ref{sec:geometry} the identification
\[
 E(K_4)=\{01,02,03,12,13,23\}
\]
with the fixed edge order
\[
 (01,02,03,12,13,23),
\]
and the corresponding incidence roots
\[
 r_a=e_1,\quad r_b=e_2,\quad r_c=e_3,\quad
 r_d=e_1-e_2,\quad r_e=e_1-e_3,\quad r_f=e_2-e_3.
\]
Thus $r_a,\ldots,r_f$ are, up to sign, the six edge vectors of the
reference realization of $K_4$ on
\[
 p_0=0,\qquad p_1=e_1,\qquad p_2=e_2,\qquad p_3=e_3.
\]

We now regard
\[
 \omega=(\omega_{01}, \omega_{02},\omega_{03},\omega_{12},\omega_{13}, \omega_{23})=(a,b,c,d,e,f)\in\mathbb R^6, \qquad a,b,c,d,e,f\geq 0
\]
as an arbitrary nonnegative weighting of the six edges and put
\begin{equation}\label{eq:affine-Z}
 Z_\omega=\sum_{\xi\in\{a,b,c,d,e,f\}}
 \left[-\frac{ \xi r_\xi}{2},
       \frac{ \xi r_\xi}{2}\right].
\end{equation}
The polynomials $D(\omega)$ and $q_i(\omega)$ are the same spanning-tree
and facet polynomials as in Section~\ref{sec:geometry}.

In the lattice--Voronoi subclass these graphical weights are precisely
the Selling conorms,
$
 \omega =\rho ,
$ and they determine the Selling Gram matrix
\[
 A(\omega)
 =\sum_{\xi\in\{a,b,c,d,e,f\}}
    \xi r_\xi r_\xi^{\mathsf T}.
\]
Equivalently, $A(\omega)$ is the reduced weighted Laplacian of $K_4$.
In the general affine problem the same weighted graphical zonotope
$Z_\omega$ is retained, while the Euclidean metric is allowed to vary independently.

\begin{theorem}\label{thm:affine-param}
Every three-dimensional parallelohedron is, up to translation, of the form
\begin{equation}\label{eq:P-LZ}
 P=LZ_\omega
\end{equation}
for some $L\in GL(3,\mathbb R)$ and some
$\omega=(a,b,c,d,e,f)\in\mathbb R^6$ with $ a,b,c,d,e,f\geq 0$ and  
$D(\omega)>0$.

For an edge $ij\in E(K_4)$, let $\overline{ij}=kl$ denote the
complementary edge, namely
\[
 \{i,j,k,l\}=\{0,1,2,3\}.
\]
After relabelling the four vertices, the five Fedorov types correspond
to the images, under the edge-complement involution
\[
 ij\longmapsto\overline{ij},
\]
of the support patterns described in \cite[Section~2]{Langi2022}.
In particular, they have respectively $6,5,4,4,3$ active graphical
generators.
\end{theorem}

\begin{proof}
L\'angi's representation of three-dimensional parallelohedra
\cite{Langi2022} shows that a translate of every such polytope can be
written as
\[
 \sum_{0\le i<j\le3}
 [0,\beta_{ij}(v_i\times v_j)],
 \qquad \beta_{ij}\ge0,
\]
where
\[
 v_0+v_1+v_2+v_3=0.
\]
The vectors $v_0,\ldots,v_3$ form a nondegenerate centered tetrahedral
configuration: every three are independent and their unique linear relation
is $v_0+v_1+v_2+v_3=0$.  The lower Fedorov types are obtained from the same
nondegenerate quadruple by setting the appropriate coefficients $\beta_{ij}$
equal to zero.

Choose the centered reference tetrahedron
\[
 u_0=-(e_1+e_2+e_3),\qquad
 u_1=e_1,\qquad
 u_2=e_2,\qquad
 u_3=e_3.
\]
Since both $(u_i)_{i=0}^3$ and $(v_i)_{i=0}^3$ satisfy the same unique
linear relation with all coefficients equal to one, there exists
$S\in GL(3,\mathbb R)$ such that
\[
 v_i=Su_i,\qquad 0\le i\le3.
\]
The covariance formula for the cross product gives
\[
 (Su_i)\times(Su_j)
 =\det(S)S^{-\mathsf T}(u_i\times u_j).
\]
Thus all tetrahedral shape variables are absorbed in the single
invertible linear map
\[
 L:=\det(S)S^{-\mathsf T}.
\]

We now compare the six cross products of the reference tetrahedron with
the incidence roots
\[
 r_{01}=e_1,\qquad
 r_{02}=e_2,\qquad
 r_{03}=e_3,\qquad
 r_{12}=e_1-e_2,\qquad
 r_{13}=e_1-e_3,\qquad
 r_{23}=e_2-e_3.
\]
A direct computation gives
\[
\begin{aligned}
 u_2\times u_3&= r_{01},&
 u_1\times u_3&=-r_{02},&
 u_1\times u_2&= r_{03},\\
 u_0\times u_3&=-r_{12},&
 u_0\times u_2&= r_{13},&
 u_0\times u_1&=-r_{23}.
\end{aligned}
\]
Equivalently,
\begin{equation}\label{eq:complement-cross-root}
 u_i\times u_j
 =\varepsilon_{ij}r_{\overline{ij}},
 \qquad
 \varepsilon_{ij}\in\{-1,1\}.
\end{equation}
Hence the coefficient attached in L\'angi's representation to the pair
$ij$ becomes the graphical weight on the complementary edge:
\begin{equation}\label{eq:beta-omega-complement}
 \omega_{\overline{ij}}:=\beta_{ij}.
\end{equation}

Using \eqref{eq:complement-cross-root}, we therefore obtain, up to
translation,
\[
\begin{aligned}
 P
 &=
 L\sum_{0\le i<j\le3}
   [0,\beta_{ij}(u_i\times u_j)]\\
 &=
 L\sum_{0\le i<j\le3}
   [0,\varepsilon_{ij}\omega_{\overline{ij}}
   r_{\overline{ij}}].
\end{aligned}
\]
For any vector $g$, the segment $[0,g]$ is a translate of
$[-g/2,g/2]$, and replacing $g$ by $-g$ does not change the latter
centered segment.  Consequently the last Minkowski sum is, up to
translation, precisely
\[
 Z_\omega
 =
 \sum_{\xi\in E(K_4)}
 \left[-\frac{\omega_\xi r_\xi}{2},
       \frac{\omega_\xi r_\xi}{2}\right].
\]
This proves
\[
 P=LZ_\omega
\]
up to translation.

Since $P$ is three-dimensional and $L$ is invertible, $Z_\omega$ is
full-dimensional.  The graphical zonotope $Z_\omega$ is
full-dimensional if and only if the support graph of $\omega$ is
connected.  By the weighted matrix--tree theorem, this is equivalent to
\[
 D(\omega)>0.
\]

Finally, \eqref{eq:beta-omega-complement} shows that the support of the
graphical weight $\omega$ is obtained from the support of the
coefficients $\beta_{ij}$ by the edge-complement involution
$ij\mapsto\overline{ij}$.  Thus the Fedorov support patterns are the
complementary-edge images of those listed in
\cite[Section~2]{Langi2022}.  Since edge complementation is a bijection
of the six edges of $K_4$, the five types have respectively
$6,5,4,4,3$ active graphical generators.
\end{proof}

For $P=LZ_\omega$, define the determinant-one metric
\begin{equation}\label{eq:affine-Q}
 Q=|\det L|^{2/3}L^{-1}L^{-\mathsf T},
 \qquad Q>0,\qquad \det Q=1.
\end{equation}
Thus $Q$ records the five Euclidean shape degrees of freedom left after
removing scale and rotations.

\begin{proposition}
\label{prop:affine-surface}
For every $D(\omega)>0$ and every $Q>0$ with $\det Q=1$,
\begin{equation}\label{eq:affine-F}
 \Iso(\omega,Q)=
 \frac{2}{D(\omega)^{2/3}}
 \sum_{i=0}^6q_i(\omega)\sqrt{z_i^{\mathsf T}Qz_i}.
\end{equation}
The Voronoi realization is the codimension-five slice
\begin{equation}\label{eq:Vor-slice-Q}
 Q_{\rm Vor}(\omega)=D(\omega)^{-1/3}A(\omega).
\end{equation}
On this slice \eqref{eq:affine-F} reduces exactly to
\eqref{eq:closed-F}.
\end{proposition}

\begin{proof}
By \eqref{eq:facet-Z-area}, the facet of $Z_\omega$ with normal $z_i$ has
area $q_i|z_i|$.  Under $L$, a planar area with normal $z_i$ is multiplied by
$|\det L|\,|L^{-\mathsf T}z_i|/|z_i|$.  From
\eqref{eq:affine-Q},
\[
 |L^{-\mathsf T}z_i|=|\det L|^{-1/3}
 \sqrt{z_i^{\mathsf T}Qz_i}.
\]
Summing opposite facets gives
\[
 \mathcal H^2(\partial P)=
 2|\det L|^{2/3}\sum_iq_i\sqrt{z_i^{\mathsf T}Qz_i}.
\]
Since $|P|=|\det L|D(\omega)$, division by $|P|^{2/3}$ proves
\eqref{eq:affine-F}.  For a lattice Voronoi cell, $L=A^{-1/2}$ and
$\det A=D$, which gives \eqref{eq:Vor-slice-Q}.
\end{proof}

\subsection{Entropy contraction estimate and affine duality}
\label{sec:affine-duality}
For the global inequality we keep the affine metric arbitrary.  The two
steps needed in the original dual formulation---the Ball--Barthe certificate
and its entropy comparison with $\cJ$---can be combined into one canonical
bridge.  This is the main simplification in the passage from lattices to
arbitrary parallelohedra.

First of all we use the same construction and notations as in Section \ref{sec:compact-entropy}.
and we prove an entropy contraction estimate. 
For a graphical edge
weight $\omega$ let $A(\omega)$ be the matrix associated with the weighted
$K_4$ as above, and let $q=q(\omega)$ be the seven facet field.  Then
\begin{equation}\label{eq:physical-adjugate}
 A_q=\sum_{i=0}^6q_i z_i z_i^{\mathsf T}
 =\operatorname{adj}A(\omega)=D(\omega)A(\omega)^{-1}.
\end{equation}
In particular $\cP(q)=D^2(\omega)$.  This is the same adjugate identity as
\eqref{eq:physical-adjugate-lattice}, now written for an arbitrary positive
edge weighting $\omega$.

For $u=(u_0, u_1, u_2, u_3, u_4,u_5, u_6)$, $u_i>0$, recall from \eqref{eq:Pdef}, \eqref{espressionel_i} 
\[
A_u=\sum_{i=0}^6u_i z_i z_i^{\mathsf T},
 \qquad
 \ell_i(u)=u_i z_i^{\mathsf T}A_u^{-1}z_i
 =u_i\partial_i\log\cP(u).
\]
By homogeneity, $\sum_i\ell_i(u)=3$.

For   $q=q(\omega)$, let $c_i$ be the corresponding cut
weights and introduce the canonical surface field
\begin{equation}\label{eq:affine-canonical-p}
 p_i:=\frac{q_i}{\sqrt{c_i}},\qquad 0\le i\le6.
\end{equation}
Then $A_p=M $ from \eqref{eq:M-J}, so
\begin{equation}\label{eq:J-as-Pp}
 \cP(p)=\det M,\qquad
 \cJ(\omega)=\frac{\cP(p)}{D^{3/2}(\omega)}.
\end{equation}
As in Section \ref{sec:compact-entropy}, we associate to $q=q(\omega)$ and to $p$ defined in \eqref{eq:affine-canonical-p} two determinantal probability measure  $\tau_q$ annd $\tau_p$.  
We need an estimate on the relative entropy of  $\tau_p$ with respect to  $\tau_q$, in terms of their one-point marginals $\ell_i(q)$ and $\ell_i(p)$. 

This estimate is provided by the next general lemma. 
\begin{lemma}\label{lem:EI}
Let \(u,v\in\mathbb{R}^7\) with $u_i, v_i\geq 0$  be such that
\(\cP(u)>0\) and \(\cP(v)>0\), and assume that
\[
\operatorname{supp}\tau_u\subseteq\operatorname{supp}\tau_v.
\]
Then
\begin{equation}\label{eq:EI}
\KL(\tau_u\Vert\tau_v)
\geq
\sum_{i=0}^6
\ell_i(u)\log\frac{\ell_i(u)}{\ell_i(v)}.
\end{equation}
\end{lemma}

\begin{proof}
Since \(\cP(u)>0\) and \(\cP(v)>0\), the matrices
\[
A_u=\sum_{i=0}^6 u_i z_i z_i^{\mathsf T},
\qquad
A_v=\sum_{i=0}^6 v_i z_i z_i^{\mathsf T}
\]
are positive definite.

Recall that, by Jacobi's formula,
\[
\ell_i(u)
=
u_i\frac{\partial}{\partial u_i}\log\cP(u)
=
u_i z_i^{\mathsf T}A_u^{-1}z_i,
\]
and similarly
\[
\ell_i(v)
=
v_i z_i^{\mathsf T}A_v^{-1}z_i.
\]

Let
\[
S:=\{i\in\{0,\ldots,6\}:\ell_i(u)>0\}.
\]
For every \(i\in S\), there exists a basis \(I\) containing \(i\) such that
\(\tau_u(I)>0\). Since
\(\operatorname{supp}\tau_u\subseteq\operatorname{supp}\tau_v\),
we also have \(\tau_v(I)>0\). Hence \(u_i>0\), \(v_i>0\), and
\(\ell_i(v)>0\) for every \(i\in S\).

For \(i\in S\), define
\[
w_i(u):=\sqrt{u_i}\,A_u^{-1/2}z_i,
\qquad
e_i(u):=\frac{w_i(u)}{\sqrt{\ell_i(u)}}.
\]
Then
\[
\|w_i(u)\|^2
=
u_i z_i^{\mathsf T}A_u^{-1}z_i
=
\ell_i(u),
\]
so that each \(e_i(u)\) is a unit vector. Moreover,
\begin{align*}
\sum_{i\in S}\ell_i(u)e_i(u)e_i(u)^{\mathsf T}
=
\sum_{i=0}^6 w_i(u)w_i(u)^{\mathsf T} =
A_u^{-1/2}
\left(\sum_{i=0}^6u_i z_i z_i^{\mathsf T}\right)
A_u^{-1/2} =I_3.
\end{align*}

Set
\[
B:=A_u^{-1/2}A_vA_u^{-1/2}.
\]
Then \(B\) is positive definite and
\[
\det B
=
\frac{\det A_v}{\det A_u}
=
\frac{\cP(v)}{\cP(u)}.
\]

By the definitions of \(\tau_u\) and \(\tau_v\),
\begin{align*}
\KL(\tau_u\Vert\tau_v)
&=
\sum_{I\in\operatorname{supp}\tau_u}
\tau_u(I)\log\frac{\tau_u(I)}{\tau_v(I)} =
\sum_{I\in\operatorname{supp}\tau_u}
\tau_u(I)
\left(
\log\frac{\cP(v)}{\cP(u)}
+
\sum_{i\in I}\log\frac{u_i}{v_i}
\right)\\
&=
\log\frac{\cP(v)}{\cP(u)}
+
\sum_{i\in S}\ell_i(u)\log\frac{u_i}{v_i} =
\log\det B
+
\sum_{i\in S}\ell_i(u)\log\frac{u_i}{v_i}.
\end{align*}

Since
\[
A_v^{-1}
=
A_u^{-1/2}B^{-1}A_u^{-1/2},
\]
for every \(i\in S\) we have
\begin{align*}
\ell_i(v)
=
v_i z_i^{\mathsf T}A_v^{-1}z_i=
v_i z_i^{\mathsf T}
A_u^{-1/2}B^{-1}A_u^{-1/2}z_i=
\frac{v_i}{u_i}\,
\ell_i(u)\,
e_i(u)^{\mathsf T}B^{-1}e_i(u).
\end{align*}
Therefore
\[
\frac{u_i\ell_i(v)}
     {v_i\ell_i(u)}
=
e_i(u)^{\mathsf T}B^{-1}e_i(u).
\]

It follows that
\begin{align*}
&\KL(\tau_u\Vert\tau_v)
-
\sum_{i=0}^6
\ell_i(u)\log\frac{\ell_i(u)}{\ell_i(v)}
\\
&\qquad=
\log\det B
+
\sum_{i\in S}\ell_i(u)
\log\left(
e_i(u)^{\mathsf T}B^{-1}e_i(u)
\right).
\end{align*}

By the spectral Jensen inequality, for every unit vector \(e\),
\[
\log(e^{\mathsf T}B^{-1}e)
\geq
e^{\mathsf T}(\log B^{-1})e.
\]
Thus
\begin{align*}
&\sum_{i\in S}\ell_i(u)
\log\left(e_i(u)^{\mathsf T}B^{-1}e_i(u)\right)
\geq
\sum_{i\in S}\ell_i(u)
e_i(u)^{\mathsf T}(\log B^{-1})e_i(u)\\
&=
\operatorname{tr}\left(
(\log B^{-1})
\sum_{i\in S}
\ell_i(u)e_i(u)e_i(u)^{\mathsf T}
\right) =
\operatorname{tr}(\log B^{-1}) =
\log\det B^{-1} =
-\log\det B.
\end{align*}
Consequently,
\[
\log\det B
+
\sum_{i\in S}\ell_i(u)
\log\left(e_i(u)^{\mathsf T}B^{-1}e_i(u)\right)
\geq 0,
\]
which proves \eqref{eq:EI}.
\end{proof}

Lemma \ref{lem:EI} is the key step to prove the following result.

\begin{theorem} 
\label{thm:canonical-affine-bridge}
For every positive graphical weight $\omega$, every $Q>0$ with $\det Q=1$,
and $q=q(\omega)$, one has
\begin{equation}\label{eq:canonical-affine-bridge}
 \frac{N_Q(q)^3}{27\cP(q)}\ge\cJ(\omega),
 \qquad
 N_Q(q):=\sum_{i=0}^6q_i\sqrt{z_i^{\mathsf T}Qz_i}.
\end{equation}
If equality holds, then, with $p$ as in
\eqref{eq:affine-canonical-p} and
$\widehat\ell_i=\ell_i(p)$, the seven quantities
\begin{equation}\label{eq:canonical-affine-equality}
 \frac{q_i\sqrt{z_i^{\mathsf T}Qz_i}}{\widehat\ell_i}
 \qquad(0\le i\le6)
\end{equation}
are equal.
\end{theorem}

\begin{proof}
Write $P_p=\cP(p)$, $P_q=\cP(q)=D^2(\omega)$ and set
\[
 \ell_i=\ell_i(q),\qquad \widehat\ell_i=\ell_i(p),
 \qquad
 \rho_i=z_i^{\mathsf T}A_p^{-1}z_i
       =\frac{\widehat\ell_i}{p_i}.
\]
The unit vectors
\[
 e_i=\frac{A_p^{-1/2}z_i}{\sqrt{\rho_i}}
\]
form the isotropic frame
$\sum_i\widehat\ell_i e_i e_i^{\mathsf T}=I_3$, where $I_3$ is the identity matrix.  Applying the
rank-one Ball--Barthe inequality to
$B=A_p^{1/2}QA_p^{1/2}$, whose determinant is $P_p$, gives
\begin{equation}\label{eq:canonical-affine-BB}
 \prod_i(z_i^{\mathsf T}Qz_i)^{\widehat\ell_i/2}
 \ge P_p^{1/2}\prod_i
 \left(\frac{\widehat\ell_i}{p_i}\right)^{\widehat\ell_i/2}.
\end{equation}
Weighted AM--GM, with weights $\widehat\ell_i/3$, followed by
\eqref{eq:canonical-affine-BB}, yields
\begin{equation}\label{eq:canonical-affine-first-step}
 \frac{N_Q(q)^3}{27P_q}
 \ge \mathcal R:=
 \frac{P_p^{1/2}}{P_q}
 \prod_i\left(\frac{q_i^2}{p_i\widehat\ell_i}
          \right)^{\widehat\ell_i/2}.
\end{equation}

It remains to compare $\mathcal R$ with $\cJ(\omega)=P_p/D^{3/2}$.
From the definition of $\mathcal R$, using $P_q=D^2(\omega)$, one obtains
\[
 2\log\frac{\mathcal R}{\cJ}
 =-\log P_p-\frac12\log P_q
  +\sum_i\widehat\ell_i
   \log\frac{q_i^2}{p_i\widehat\ell_i}.
\]
On the other hand,
\[
 \KL(\tau_p\Vert\tau_q)
 -\sum_i\widehat\ell_i\log\frac{\widehat\ell_i}{\ell_i}
 =\log\frac{P_q}{P_p}
  +\sum_i\widehat\ell_i
   \log\frac{p_i\ell_i}{q_i\widehat\ell_i}.
\]
Since $p_i/q_i=c_i^{-1/2}$, $c_i=D\ell_i/q_i$, and
$\sum_i\widehat\ell_i=3$, the two displayed expressions agree.  Therefore
\begin{equation}\label{eq:canonical-affine-entropy}
 2\log\frac{\mathcal R}{\cJ}
 =\KL(\tau_p\Vert\tau_q)
 -\sum_i\widehat\ell_i
   \log\frac{\widehat\ell_i}{\ell_i}.
\end{equation}
Lemma~\ref{lem:EI}, applied with $u=p$ and $v=q$, shows that the
right-hand side is nonnegative.  Combining
\eqref{eq:canonical-affine-first-step} and
\eqref{eq:canonical-affine-entropy} proves
\eqref{eq:canonical-affine-bridge}.  Equality in the final bound forces
equality in the weighted AM--GM step, which is precisely
\eqref{eq:canonical-affine-equality}.
\end{proof}

\begin{theorem} 
\label{thm:affine-positive-global}
Let $\omega\in\mathbb R^6$ with $\omega_i>0$ and let $Q>0$, $\det Q=1$.  Then
\begin{equation}\label{eq:affine-positive-global}
 \Iso(\omega,Q)\ge F_{\BCC}.
\end{equation}
Equality holds only for the regular truncated octahedron, up to similarity.
\end{theorem}

\begin{proof}
Let $q=q(\omega)$ and
$N_Q=\sum_iq_i\sqrt{z_i^{\mathsf T}Qz_i}$.  Since
$\cP(q)=D(\omega)^2$, Proposition~\ref{prop:affine-surface} gives
\[
 \left(\frac{\Iso(\omega,Q)}6\right)^3
 =\frac{N_Q^3}{27\cP(q)}.
\]
Apply the canonical affine bridge
Theorem~\ref{thm:canonical-affine-bridge}, and then the sharp lattice
determinant inequality Theorem~\ref{thm:J-global}:
\begin{equation}\label{eq:full-affine-chain}
 \left(\frac{\Iso(\omega,Q)}6\right)^3
 \ge\cJ(\omega) 
 \ge \cJ_{\BCC}.
\end{equation}
Since $F_{\BCC}=6\cJ_{\BCC}^{1/3}$, this proves
\eqref{eq:affine-positive-global}.

If equality holds in the final conclusion, then $\cJ=\cJ_{\BCC}$.
The equality statement of Theorem~\ref{thm:J-global} forces $\omega$ onto
the BCC ray.  Normalize
$q=(3,3,3,3;1,1,1)$.  Then
$c=(3,3,3,3;4,4,4)$ and hence
\[
 p=(\sqrt3,\sqrt3,\sqrt3,\sqrt3;1/2,1/2,1/2)
   =\frac1{\sqrt3}(3,3,3,3;\sqrt3/2,\sqrt3/2,\sqrt3/2).
\]
Leverage scores are invariant under a common scaling of the field, so direct
substitution gives two leverage values
\[
 \ell_S=\frac9{11}-\frac{3\sqrt3}{22}
 \quad(i<4),
 \qquad
 \ell_D=-\frac1{11}+\frac{2\sqrt3}{11}
 \quad(i\ge4),
\]
with
\[
 \frac{3\ell_D}{\ell_S}=\frac2{\sqrt3}.
\]
The equality condition \eqref{eq:canonical-affine-equality} therefore forces
the seven metric
lengths $w_i=\sqrt{z_i^{\mathsf T}Qz_i}$ to satisfy
\[
 w_0=w_1=w_2=w_3=:S,
 \qquad
 w_4=w_5=w_6=\frac{2S}{\sqrt3}.
\]
Since $w_1^2,w_2^2,w_3^2$ are the diagonal entries of $Q$, they are all
$S^2$.  The three double-cut equalities give all off-diagonal entries equal
to $-S^2/3$; equivalently
\[
 Q=\frac{S^2}{3}
 \begin{pmatrix}3&-1&-1\\-1&3&-1\\-1&-1&3\end{pmatrix}.
\]
The condition $\det Q=1$ fixes the scalar.  This is exactly the regular BCC
metric, so the resulting parallelohedron is similar to the regular truncated
octahedron.  The converse is immediate.
\end{proof}

\subsection{The truncated octahedron theorem}
\label{sec:affine-boundary}

The strata with at most four generators are already controlled in the full
zonotopal class by Jo\'os--L\'angi \cite{JoosLangi2023}.  The only new boundary
case is the five-generator elongated rhombic dodecahedron.  The following
result extends Theorem~\ref{thm:fivefacet} from the Voronoi slice
\[
 Q=Q_{\rm Vor}(\omega)=D(\omega)^{-1/3}A(\omega)
\]
to an arbitrary determinant-one affine metric $Q$.  The same four-point
determinant inequality from Section~\ref{sec:FCC} remains the key input, but
the Voronoi metric relation is no longer imposed.

\begin{theorem} 
\label{thm:five-generator-affine}
Every three-dimensional parallelohedron generated by five segments satisfies
\begin{equation}\label{eq:five-affine-bound}
 \Iso(P)>3\,2^{5/6}.
\end{equation}
On the closure of the five-generator stratum,
\begin{equation}\label{eq:five-affine-closure}
 \Iso(P)\ge3\,2^{5/6},
\end{equation}
with equality exactly at the regular rhombic dodecahedron, up to similarity.
\end{theorem}

\begin{proof}
By tetrahedral symmetry set $a=0$ and $f=t$.  Relabel the six active facet
polynomials and cut weights of Section~\ref{sec:FCC} as follows:
\begin{align*}
&C_1=de+t(d+e), &&L_1=b+c,\\
&C_2=bc+t(b+c), &&L_2=d+e,\\
&C_3=ce, &&L_3=b+d+t,\\
&C_4=bd, &&L_4=c+e+t,\\
&C_5=be, &&L_5=c+d+t,\\
&C_6=cd, &&L_6=b+e+t.
\end{align*}
Thus
\[
 (C_1,\ldots,C_6)=(q_0,q_1,q_2,q_3,q_5,q_6),
 \qquad
 (L_1,\ldots,L_6)=(c_0,c_1,c_2,c_3,c_5,c_6).
\]
For a six-tuple $x=(x_1,\ldots,x_6)$ in this subsection, write
\[
 \Gamma_\ast(x_1,\ldots,x_6):=\Gamma(d),
 \qquad
 (d_{01},d_{23},d_{02},d_{13},d_{12},d_{03})
 =(x_1,x_2,x_3,x_4,x_5,x_6).
\]
This convention records explicitly the nonstandard edge order used below.
Let
\[
 \mathsf G_\omega=\Gamma_\ast(L_1,\ldots,L_6).
\]
As in \eqref{eq:facet-structure},
\begin{equation}\label{eq:affine-five-G}
 \det\mathsf G_\omega=D,
 \qquad
 \partial_{L_i}\det\mathsf G_\omega=C_i.
\end{equation}

Now let $Q>0$, $\det Q=1$ be an arbitrary affine metric.  For the six active
cut directions $z_0,z_1,z_2,z_3,z_5,z_6$, set
\[
 (R_1,\ldots,R_6)=
 (z_0^{\mathsf T}Qz_0,z_1^{\mathsf T}Qz_1,
 z_2^{\mathsf T}Qz_2,z_3^{\mathsf T}Qz_3,
 z_5^{\mathsf T}Qz_5,z_6^{\mathsf T}Qz_6)
\]
and use the edge order
\[
 d_{01}^2=R_1,\quad d_{23}^2=R_2,\quad
 d_{02}^2=R_3,\quad d_{13}^2=R_4,\quad
 d_{12}^2=R_5,\quad d_{03}^2=R_6.
\]
These are squared Euclidean distances: if
$p_1=z_0$, $p_2=z_2$, $p_3=z_6$ and $p_0=0$, then the remaining differences
are $p_1-p_2=z_5$, $p_1-p_3=z_3$, and $p_2-p_3=-z_1$.  Hence
\begin{equation}\label{eq:Gamma-R-congruence}
 \Gamma_\ast(R_1,\ldots,R_6)=U^{\mathsf T}QU,
 \qquad U=(z_0\ z_2\ z_6),
 \qquad |\det U|=1,
\end{equation}
so
\begin{equation}\label{eq:Gamma-R-det1}
 \det\Gamma_\ast(R_1,\ldots,R_6)=1.
\end{equation}
Put
\[
 \mathsf H_Q=\Gamma_\ast(\sqrt{R_1},\ldots,\sqrt{R_6}).
\]
By linearity of $\Gamma$ and Jacobi's formula,
\begin{equation}\label{eq:affine-five-N}
 N_Q:=\sum_{i=1}^6C_i\sqrt{R_i}
 =\operatorname{tr}(\operatorname{adj}\mathsf G_\omega\,\mathsf H_Q).
\end{equation}
If $\mu_1,\mu_2,\mu_3$ are the eigenvalues of
$\mathsf G_\omega^{-1/2}\mathsf H_Q\mathsf G_\omega^{-1/2}$, then
\[
 N_Q=D(\mu_1+\mu_2+\mu_3),
 \qquad
 \det\mathsf H_Q=D\mu_1\mu_2\mu_3.
\]
Therefore
\begin{equation}\label{eq:affine-five-spectral}
 N_Q^3\ge27D^2\det\mathsf H_Q.
\end{equation}
Theorem~\ref{thm:fourpoint}, applied to the metric $d$, and
\eqref{eq:Gamma-R-det1} give
\[
 2(\det\mathsf H_Q)^2\ge1,
 \qquad
 \det\mathsf H_Q\ge2^{-1/2}.
\]
Since the affine quotient is $2N_Q/D^{2/3}$, we conclude
\[
 \Iso(P)\ge6\,2^{-1/6}=3\,2^{5/6}.
\]

If equality holds, Theorem~\ref{thm:fourpoint} forces all six $d_{ij}$ to be
equal, equivalently all six $R_i=d_{ij}^2$ are equal.  Equality in
\eqref{eq:affine-five-spectral} forces
$\mathsf H_Q=\lambda\mathsf G_\omega$ for some $\lambda>0$.  Since
\[
 \mathsf H_Q=\Gamma_\ast(\sqrt{R_1},\ldots,\sqrt{R_6}),
 \qquad
 \mathsf G_\omega=\Gamma_\ast(L_1,\ldots,L_6),
\]
and $\Gamma$ is injective on the six edge variables, it follows that
\[
 \sqrt{R_i}=\lambda L_i\qquad(1\le i\le6).
\]
As the $R_i$ are all equal, all six $L_i$ are equal.  The relations
$L_3=L_5$, $L_3=L_6$, $L_1=L_2$, and $L_1=L_3$ yield
$b=c$, $d=e$, $b=d$, and finally $t=0$.  Thus the graphical weights lie on
the four-generator FCC ray $(0,m,m,m,m,0)$.  Together with the equality of
all six distances $d_{ij}$, this forces the determinant-one metric to be the
tetrahedrally symmetric FCC metric.  Hence the limiting polytope is the
regular rhombic dodecahedron, up to similarity.  Equality is not attained in
the open five-generator stratum.
\end{proof}

\begin{corollary} 
\label{cor:all-boundary-affine}
Every three-dimensional parallelohedron which is not of truncated-octahedral
type satisfies
\begin{equation}\label{eq:all-boundary-affine}
 \Iso(P)\ge3\,2^{5/6}>F_{\BCC},
\end{equation}
with equality in the first inequality only for the regular rhombic
dodecahedron.
\end{corollary}

\begin{proof}
The three- and four-generator strata follow from
\cite[Theorem~3]{JoosLangi2023}; Theorem~\ref{thm:five-generator-affine}
handles the remaining five-generator type.
\end{proof}

\begin{proof}[Proof of Theorem~\ref{thm:main}]
The sharp estimate for the non-truncated Fedorov types, together with its
equality case, is Corollary~\ref{cor:all-boundary-affine}.  It remains to
prove the global BCC bound.

By Theorem~\ref{thm:affine-param}, every three-dimensional parallelohedron is
represented by a connected nonnegative edge weighting $\omega$ of $K_4$ and
a free affine metric.  If all six graphical generators are active,
Theorem~\ref{thm:affine-positive-global} gives the BCC bound and its equality
classification.  If at least one generator is absent, the polytope belongs to
a non-truncated Fedorov stratum and Corollary~\ref{cor:all-boundary-affine}
gives
\[
 \Iso(P)\ge3\,2^{5/6}>F_{\BCC}.
\]
Thus the regular truncated octahedron is the unique global minimizer up to
similarity, and Theorem~\ref{thm:main} follows.
\end{proof}

\section{Open problems}
\label{sec:open-problems}

For a full-rank lattice $\Lambda\subset\mathbb R^n$, define
\[
 \Iso_n(\Vor(\Lambda))=
 \frac{\mathcal H^{n-1}(\partial\Vor(\Lambda))}
      {\mathcal H^{n}(\Vor(\Lambda))^{(n-1)/n}},
 \qquad
 \gamma_n=\inf_\Lambda\Iso_n(\Vor(\Lambda)).
\]
The planar problem is minimized by the hexagonal lattice
\cite{FejesToth1964}, and Theorem~\ref{thm:BCC-global} determines
$\gamma_3$.  Already in dimension
four, two special features of the present proof disappear.  Not every lattice
is of Voronoi's first kind, so there is no global nonnegative-conorm chart of
the type used here~\cite{McKilliamGrantClarkson2014}; and the combinatorial
stratification grows rapidly.  There are $52$ four-dimensional types of
parallelohedra~\cite{Engel1992}, while the classification of lattice
Voronoi polytopes in dimension five contains $110244$ affine
types~\cite{DutourSikiric2016}.

A simple comparison also shows that the sequence $A_n^*$ cannot be adopted
uncritically as the higher-dimensional answer.  The Voronoi cell of $D_4$ is
a regular $24$-cell~\cite[Chap.~4]{ConwaySloane1999}.  In the normalization
with edge length $1$ it has volume $2$ and $24$ regular-octahedral facets,
each of volume $\sqrt2/3$.  Hence
\[
 \Iso_4(\Vor(D_4))=\frac{24(\sqrt2/3)}{2^{3/4}}=2^{11/4}
 \approx6.727171322.
\]
The Voronoi cell of $A_4^*$ is, up to rescaling, the graphical zonotope of
$K_5$~\cite[Chap.~21]{ConwaySloane1999}.  In the root-lattice realization,
the matrix-tree formula gives volume $5^3\sqrt5=125\sqrt5$.  There are ten
facets of cut type $1|4$, each a $K_4$ graphical zonotope of volume $32$, and
twenty facets of type $2|3$, each the orthogonal product of a segment of
length $\sqrt2$ and a $K_3$ graphical zonotope of area $3\sqrt3$, hence of
volume $3\sqrt6$.  Thus
$\mathcal H^4(\Vor(A_4^*))=125\sqrt5$,
$\mathcal H^3(\partial\Vor(A_4^*))=320+60\sqrt6$,
and therefore
\[
 \Iso_4(\Vor(A_4^*))=
 \frac{320+60\sqrt6}{(125\sqrt5)^{3/4}}
 =\frac{4\,5^{3/8}}{25}(16+3\sqrt6)
 \approx6.831123656,
\]
hence $D_4$ is strictly better than $A_4^*$.

\begin{openproblem}[The four-dimensional case]
Determine $\gamma_4$ and classify all minimizers.  Is the regular $24$-cell,
that is the Voronoi cell of $D_4$, the unique minimizer up to similarity?
A complete proof would have to control all four-dimensional
combinatorial types together with their metric cones, or replace the
stratum-by-stratum analysis by a genuinely higher-dimensional invariant
inequality.
\end{openproblem}

\begin{openproblem}[Lattices of Voronoi's first kind]
For general $n$, restrict to lattices admitting an obtuse superbase.  Their
Voronoi-relevant vectors are indexed by the $2^n-1$ nontrivial cuts of
$K_{n+1}$, so a higher-rank analogue of the determinant and entropy framework
is available. Determine the minimizer of $\Iso_n$ in this
subclass.  In particular, for which $n$'s is $A_n^*$ optimal?
\end{openproblem}

\begin{openproblem}[Quantitative stability]
Find a distance on the scale-free moduli of three-dimensional parallelohedra and the sharp constant in a global estimate of the form
\[
 \Iso(P)-\Iso_\BCC
 \ge c\,\operatorname{dist}(P,\BCC)^2.
\]
\end{openproblem}

\begin{openproblem}[Beyond parallelohedra]
Theorem~\ref{thm:main} settles the isoperimetric problem
inside the class of three-dimensional parallelohedra.  Determine the
corresponding minimum among all equal-volume translative fundamental domains,
or among periodic equal-cell partitions (such a minimizer exists due to   the periodic existence
theory in~\cite{CesaroniNovaga2024,CesaroniNovaga2025}).
\end{openproblem}

\section*{Acknowledgements}
The authors are members of the Gruppo Nazionale per l'Analisi Matematica, la Probabilità e le loro Applicazioni (GNAMPA) of the Istituto Nazionale di Alta Matematica (INdAM).

\end{document}